\documentclass[a4paper, 11pt,]{article}

\usepackage{amsmath,amsthm}
\usepackage{amsfonts,amssymb}
\usepackage{graphicx}
\usepackage{calc}
\usepackage{multicol}
\def\normo#1{\left\|#1\right\|}

\def\aabs#1{\big|#1\big|}
\def\brk#1{\left(#1\right)}

\def\rev#1{\frac{1}{#1}}
\def\half#1{\frac{#1}{2}}
\def\norm#1{\|#1\|}
\def\jb#1{\langle#1\rangle}

\newcommand{\N}{{\mathbb N}}
\newcommand{\T}{{\mathbb T}}

\newcommand{\R}{{\mathbb R}}
\newcommand{\C}{{\mathbb C}}
\newcommand{\Z}{{\mathbb Z}}
\newcommand{\ft}{{\mathcal{F}}}

\newcommand{\les}{{\lesssim}}
\newcommand{\ges}{{\gtrsim}}
\newcommand{\ra}{{\rightarrow}}

\newcommand{\Sch}{{\mathcal{S}}}
\newcommand{\supp}{{\mbox{supp}}}

\numberwithin{equation}{section}

\newtheorem{theorem}{Theorem}[section]

\newtheorem{proposition}[theorem]{Proposition}

\newtheorem{remark}[theorem]{Remark}

\begin{document}
\title{\bf Global well posedness and inviscid limit for the  Korteweg-de
Vries-Burgers equation}

\author{\bf Zihua Guo and Baoxiang Wang \date{}\\
{\small \it
LMAM, School of Mathematical Sciences, Peking University, Beijing 100871, China}\\
{\small E-mails: zihuaguo, wbx@math.pku.edu.cn} }\maketitle

\maketitle


{\bf Abstract:} Considering the Cauchy problem for the Korteweg-de
Vries-Burgers equation
\begin{eqnarray*}
u_t+u_{xxx}+\epsilon |\partial_x|^{2\alpha}u+(u^2)_x=0,\ \
u(0)=\phi,
\end{eqnarray*}
where $0<\epsilon,\alpha\leq 1$ and $u$ is a real-valued function,
we show that it is globally well-posed in $H^s\ (s>s_\alpha)$, and
uniformly globally well-posed in $H^s\ (s>-3/4)$ for all $\epsilon
\in (0,1]$. Moreover, we prove that for any $T>0$, its solution
converges in $C([0,T]; \,H^s)$ to that of the KdV equation if
$\epsilon$ tends to $0$.



\section{Introduction}
In this paper, we study the Cauchy problem for the Korteweg-de
Vries-Burgers (KdV-B) equation with fractional dissipation
\begin{eqnarray}\label{eq:kdvb}
u_t+u_{xxx}+\epsilon |\partial_x|^{2\alpha}u+(u^2)_x=0, \ \
u(0)=\phi,
\end{eqnarray}
where $0<\epsilon, \alpha \leq 1$, $u$ is a real-valued function of
$(x, t) \in \mathbb{R}\times \mathbb{R}_+$. Eq. \eqref{eq:kdvb} has
been derived as a model for the propagation of weakly nonlinear
dipersive long waves in some physical contexts when dissipative
effects occur (cf. \cite{OS}). The global well-posedness of
\eqref{eq:kdvb} and the generalized KdV-Burgers equation has been
studied by many authors (see \cite{MR,MR2} and the reference
therein).

In \cite{MR} Molinet and Ribaud studied Eq. \eqref{eq:kdvb} in the
case $\alpha=1$ and showed that \eqref{eq:kdvb} is globally
well-posed in $H^{s}\ (s>-1)$. The main tool used in \cite{MR} is an
$X^{s,b}$-type space which contains the dissipative structure. Their
result is sharp in the sense that the solution map of
\eqref{eq:kdvb} fails to be $C^2$ smooth at $t=0$ if $s<-1$. In
particular, one can't get lower regularity simply using fixed-point
machinery. Note that $s=-1$ is lower than the critical index
$s=-3/4$ for the KdV equation and also lower than the critical index
$s=-1/2$ for the dissipative Burgers equation. The case $0<\alpha
<1$ was left open and it was conjectured in \cite{MR} that one can
get that \eqref{eq:kdvb} is globally well-posed in $H^s\
(s>s_c=(\alpha-3)/2(2-\alpha))$ by using the same strategy as
$\alpha=1$.

In the first part of this paper, we will study the global well
posedness of Eq. \eqref{eq:kdvb} by following some ideas in
\cite{MR}\footnote{After the paper was finished, the authors were
noted that the same results in this part were also obtained by
St\'{e}phane Vento \cite{Vento} using the similar method.}. The main
issue reduces to a bilinear estimate
\begin{equation}\label{eq:bi1}
\norm{\partial_x(uv)}_{X^{-1/2+\delta, s, \alpha}}\leq
C\norm{u}_{X^{1/2, s, \alpha}}\norm{v}_{X^{1/2, s, \alpha}}.
\end{equation}
For the definition of $X^{b, s, \alpha}$, one can refer to
\eqref{eq:Xs} below. We will apply the $[k;Z]$-multiplier method in
\cite{Tao} to prove \eqref{eq:bi1}. We obtain a critical number
\begin{eqnarray}\label{eq:sa}
s_\alpha=\left \{
\begin{array}{ll}
-3/4, &  0<\alpha\leq1/2,\\
-3/(5-2\alpha),&  1/2<\alpha\leq 1.
\end{array}
\right.
\end{eqnarray}
It is worth to note that $s_\alpha$ is strictly bigger than the
conjectured number $s_c$ for $0<\alpha<1$. We prove that
\eqref{eq:bi1} holds if  and only if $s>s_\alpha$.  So, it seems
that $s>s_\alpha$ is an essential limitation of this method.

In the second part of this paper, we study the inviscid limit
behavior of \eqref{eq:kdvb} when $\epsilon$ goes to $0$. Formally,
if $\epsilon=0$ then \eqref{eq:kdvb} reduces to the KdV equation
\begin{eqnarray}\label{eq:kdv}
u_t+u_{xxx}+(u^2)_x=0, \ \ u(0)=\phi.
\end{eqnarray}
The local well posedness of Eq. \eqref{eq:kdv} in $L^2$ was
established by Bourgain \cite{Bour} and the $X^{b,s}$-theory was
discovered. This local solution is a global one by using the
conservation of $L^2$ norm. The optimal result on local
well-posedness in $H^s$ was obtained by Kenig, Ponce, Vega
\cite{KPV}, where they developed the sharp bilinear estimates and
obtained that \eqref{eq:kdv} is locally well-posed for $s>-3/4$. The
sharp result on global well-posedness in $H^s$ was obtained  in
\cite{Tao2}, it was shown that \eqref{eq:kdv} is globally well-posed
in $H^s$ for $s>-3/4$, where a kind of modified energy method, so
called I-method, is introduced.

A natural question is whether the solution of \eqref{eq:kdvb}
converges to that of \eqref{eq:kdv} if $\epsilon$ goes to $0$. We
will prove that the global solution of \eqref{eq:kdvb} converges to
the solution of \eqref{eq:kdv} as $\epsilon\to 0$ in the natural
space $C([0,T],H^s)$ for $-3/4<s\leq 0$. To achieve this, we need to
control the solution uniformly in $\epsilon$, which is independent
of the properties of dissipative term. We prove a uniform global
well-posedness result using $l^1$-variant $X^{b,s}$-type space and
the I-method. Notice that \eqref{eq:kdvb} is invariant under the
following scaling for $0<\lambda\leq 1$
\begin{equation}
u(x,t)\rightarrow \lambda^2 u(\lambda x, \lambda^3 t),\
\phi(x)\rightarrow \lambda^2\phi(\lambda x), \ \epsilon\rightarrow
\lambda^{3-2\alpha}\epsilon.
\end{equation}
The equation \eqref{eq:kdvb} has less symmetries than the KdV
equation \eqref{eq:kdv} due to the dissipative term. Hence the
proofs for the pointwise estimate of the multipliers in our argument
are different from those in the KdV equation \cite{Tao2}. The basic
idea is the same,  and to exploit dedicated cancelation to remove
the singularity in the denominator.

For the limit behavior, we need to study the difference equation
between \eqref{eq:kdvb} and \eqref{eq:kdv}. We first treat the
dissipative term as perturbation and then use the uniform Lipschitz
continuity property of the solution map. Similar idea can be found
in \cite{Wang} for the inviscid limit of the complex Ginzburg-Landau
equation. For $T>0$, we denote $S_T^{\epsilon}$, $S_T$ the solution
map of \eqref{eq:kdvb}, \eqref{eq:kdv} respectively. Now we state
our main results. The notations used in this paper can be found in
Section \ref{notation}.

\begin{theorem}\label{t11}
Assume $0<\epsilon, \alpha\leq 1$. Let $s_\alpha$ be given in
\eqref{eq:sa}. Let $\phi \in H^s(\R)$, $s>s_\alpha$. For any $T>0$,
there exists a unique solution $u_\epsilon$ of \eqref{eq:kdvb} in
\begin{equation}
Z_T=C([0,T],H^s)\cap X^{1/2,s,\alpha}_T.
\end{equation}
Moreover, the solution map $S_T^\epsilon: \phi\rightarrow u$ is
smooth from $H^s(\R)$ to $Z_T$ and $u$ belongs to
$C((0,\infty),H^{\infty}(\R))$.
\end{theorem}

Notice that the critical regularity for the fractional Burgers
equation is $s=3/2 -2\alpha$ in the sense of scaling. Thus if
$1/2<\alpha\leq 1$ then $s_\alpha$ is lower than the critical
regularity for the KdV and also for the fractional Burgers equation.
In the proof we need to exploit the properties of the dissipative
term both in bilinear estimates and regularity for the solution.
Therefore, the results in Theorem \ref{t11} depend on $\epsilon>0$.
For the uniform well-posedness, we have the following,

\begin{theorem}\label{t12}
Assume $0<\alpha\leq 1$ and $-3/4<s\leq 0$. Let $\phi \in H^s(\R)$.
Then for any $T>0$, the solution map $S_T^\epsilon$ in Theorem
\ref{t11} satisfies for all $0<\epsilon \leq 1$
\begin{equation}
\norm{S_T^\epsilon \phi}_{F^s(T)}\les C(T,\norm{u}_{H^s})
\end{equation}
where $F^s(T)\subset C([0,T];H^s)$ which will be defined later and
$C(\cdot,\cdot)$ is a continuous function with $C(\cdot,0)=0$, and
also satisfies that for all $0<\epsilon \leq 1$
\begin{eqnarray}
\norm{S_T^\epsilon (\phi_1)-S_T^\epsilon
(\phi_2)}_{C([0,T],H^s)}\leq
C(T,\norm{\phi_1}_{H^s},\norm{\phi_2}_{H^s})\norm{\phi_1-\phi_2}_{H^s}.
\end{eqnarray}
\end{theorem}

We also have the uniform persistence of regularity, following the
standard argument. The similar conclusions in Theorem \ref{t12} also
hold for the complex-valued equation \eqref{eq:kdvb} for a small
$T=T(\norm{u}_{H^s})>0$. Our final result is on the limit behavior.
\begin{theorem}\label{t13}
Assume $0<\alpha\leq 1$. Let $\phi \in H^s(\R)$, $-3/4<s\leq 0$. For
any $T>0$, then
\begin{equation}\label{eq:limitthm}
\lim_{\epsilon\rightarrow
0^+}\norm{S_T^\epsilon(\phi)-S_T(\phi)}_{C([0,T],H^s)}=0.
\end{equation}
\end{theorem}

\begin{remark} \rm
We are only concerned with the limit in the same regularity space.
There seems no convergence rate. This can be seen from the linear
solution,
\begin{equation}
\norm{e^{-t\partial_{x}^3-t\epsilon|\partial_x|^{2\alpha}}\phi-e^{-t\partial_{x}^3}\phi}_{C([0,T],H^s)}\rightarrow
0, \quad \mbox{as }\epsilon\rightarrow 0,
\end{equation}
but without any convergence rate. We believe that there is a
convergence rate if we assume the initial data has higher regularity
than the limit space. For example, we prove that
\begin{equation}
\norm{S_T^\epsilon(\phi_1)-S_T(\phi_2)}_{C([0,T],L^2)}\les
\norm{\phi_1-\phi_2}_{L^2}+\epsilon^{1/2}C(T,\norm{\phi_1}_{H^1},\norm{\phi_2}_{L^2}).
\end{equation}

We only prove our results in the case $s\leq 0$ and our method also
works for $s>0$. For the complex valued equation \eqref{eq:kdvb},
the limit behavior \eqref{eq:limitthm} holds for a small
$T=T(\norm{\phi}_{H^s})>0$.
\end{remark}

The rest of the paper is organized as following. In Section 2 we
present some notations and Banach function spaces. The proof of
Theorem \ref{t11} is given in Section 3. We present uniform LWP in
Section 4 and prove Theorem \ref{t12} in Section 5. Theorem
\ref{t13} is proved in Section 6.

\section{Notation and Definitions} \label{notation}
For $x, y\in \R$, $x\sim y$ means that there exist $C_1, C_2 > 0$
such that $C_1|x|\leq |y| \leq C_2|x|$. For $f\in \Sch'$ we denote
by $\widehat{f}$ or $\ft (f)$ the Fourier transform of $f$ for both
spatial and time variables,
\begin{eqnarray*}
\widehat{f}(\xi, \tau)=\int_{\R^2}e^{-ix \xi}e^{-it \tau}f(x,t)dxdt.
\end{eqnarray*}
We denote  by $\ft_x$ the the Fourier transform on spatial variable
and if there is no confusion, we still write $\ft=\ft_x$. Let
$\mathbb{Z}$ and $\mathbb{N}$ be the sets of integers and natural
numbers, respectively. $\Z_+=\N \cup \{0\}$. For $k\in \Z_+$ let
\[{I}_k=\{\xi: |\xi|\in [2^{k-1}, 2^{k+1}]\}, \ k\geq 1; \quad I_0=\{\xi: |\xi|\leq 2\}.\]
Let $\eta_0: \R\rightarrow [0, 1]$ denote an even smooth function
supported in $[-8/5, 8/5]$ and equal to $1$ in $[-5/4, 5/4]$. For
$k\in \N$ let $\eta_k(\xi)=\eta_0(\xi/2^k)-\eta_0(\xi/2^{k-1})$ and
$\eta_{\leq k}=\sum_{k'=0}^k\eta_{k'}$. For $k\in \Z$ let
$\chi_k(\xi)=\eta_0(\xi/2^k)-\eta_0(\xi/2^{k-1})$. Roughly speaking,
$\{\chi_k\}_{k\in \mathbb{Z}}$ is the homogeneous decomposition
function sequence and $\{\eta_k\}_{k\in \mathbb{Z}_+}$ is the
non-homogeneous decomposition function sequence to the frequency
space.

For $k\in \Z_+$ let $P_k$ denote the operator on $L^2(\R)$ defined
by
\[
\widehat{P_ku}(\xi)=\eta_k(\xi)\widehat{u}(\xi).
\]
By a slight abuse of notation we also define the operator $P_k$ on
$L^2(\R\times \R)$ by the formula $\ft(P_ku)(\xi,
\tau)=\eta_k(\xi)\ft (u)(\xi, \tau)$. For $l\in \Z$ let
\[
P_{\leq l}=\sum_{k\leq l}P_k, \quad P_{\geq l}=\sum_{k\geq l}P_k.
\]

We define the Lebesgue spaces $L_T^qL_x^p$ and $L_x^pL_T^q$ by the
norms
\begin{equation}
\norm{f}_{L_T^qL_x^p}=\normo{\norm{f}_{L_x^p}}_{L_t^q([0,T])}, \quad
\norm{f}_{L_x^pL_T^q}=\normo{\norm{f}_{L_t^q([0,T])}}_{L_x^p}.
\end{equation}
We denote by $W_0$ the semigroup associated with Airy-equation
\[
\ft_x(W_0(t)\phi)(\xi)=\exp[i\xi^3t]\widehat{\phi}(\xi), \  \forall
\ t\in \R,\ \phi \in \mathcal {S}'.
\]
For $0<\epsilon\leq 1$ and $0< \alpha \leq 1$, we denote by
$W_\epsilon^\alpha$ the semigroup associated with the free evolution
of \eqref{eq:kdvb},
\[
\ft_x(W_\epsilon^\alpha(t)\phi)(\xi)=\exp[-\epsilon
|\xi|^{2\alpha}t+i\xi^3t]\widehat{\phi}(\xi),\ \forall \ t\geq 0,\
\phi \in \mathcal {S}',
\]
and we extend $W_\epsilon^\alpha$ to a linear operator defined on
the whole real axis by setting
\[
\ft_x(W_\epsilon^\alpha(t)\phi)(\xi)=\exp[-\epsilon
|\xi|^{2\alpha}|t|+i\xi^3t]\widehat{\phi}(\xi),\ \forall \ t\in \R,
\ \phi \in \mathcal {S}'.
\]

To study the low regularity of \eqref{eq:kdvb}, Molinet and Ribaud
introduce the variant version of Bourgain's spaces with dissipation
\begin{equation}\label{eq:Xs}
\norm{u}_{X^{b,s,\alpha}}=\norm{\jb{i(\tau-\xi^3)+|\xi|^{2\alpha}}^b\jb{\xi}^s\widehat{u}}_{L^2(\R^2)},
\end{equation}
where $\jb{\cdot}=(1+|\cdot|^2)^{1/2}$. The standard $X^{b,s}$ space
for \eqref{eq:kdv} used by Bourgain \cite{Bour} and Kenig, Ponce,
Vega \cite{KPV} is defined by
\begin{eqnarray*}
\norm{u}_{X^{b,s}}=\norm{\jb{\tau-\xi^3}^b\jb{\xi}^s\widehat{u}}_{L^2(\R^2)}.
\end{eqnarray*}
The space $X^{1/2,s,\alpha}$ turns out to be very useful to capture
both dispersive and dissipative effect. From the technical level,
the dissipation will give bounds below for the modulations. These
bounds will weaken the frequency interaction for $\alpha> 1/2$, but
won't for $\alpha\leq 1/2$.

In order to study the uniform global wellposedness for
\eqref{eq:kdvb} and the limit behavior, we use an $l^1$ Besov-type
norm of $X^{b,s}$. For $k\in \Z_+$ we define the dyadic
$X^{b,s}$-type normed spaces $X_k=X_k(\R^2)$,
\begin{eqnarray*}
X_k=\{f\in L^2(\R^2): &&f(\xi,\tau) \mbox{ is supported in }
I_k\times\R \mbox{ and }\\&& \norm{f}_{X_k}=\sum_{j=0}^\infty
2^{j/2}\norm{\eta_j(\tau-\xi^3)\cdot f}_{L^2}\}.
\end{eqnarray*}
Structures of this kind of spaces were introduced, for instance, in
\cite{Tata}, \cite{IKT} and \cite{In-Ke} for the BO equation. From
the definition of $X_k$, we see that for any $l\in \Z_+$ and $f_k\in
X_k$ (see also \cite{IKT}),
\begin{equation}
\sum_{j=0}^\infty 2^{j/2}\norm{\eta_j(\tau-\xi^3)\int
|f_k(\xi,\tau')|2^{-l}(1+2^{-l}|\tau-\tau'|)^{-4}d\tau'}_{L^2}\les
\norm{f_k}_{X_k}.
\end{equation}
Hence for any $l\in \Z_+$, $t_0\in \R$, $f_k\in X_k$, and $\gamma
\in \Sch(\R)$, then
\begin{equation}
\norm{\ft[\gamma(2^l(t-t_0))\cdot \ft^{-1}f_k]}_{X_k}\les
\norm{f_k}_{X_k}.
\end{equation}
For $-3/4<s\leq 0$, we define the following spaces:
\begin{eqnarray}
&&F^{s}= \{u\in \mathcal {S}'(\R^2):\ \norm{u}_{F^s}^2=\sum_{k \in
\Z_+}2^{2sk}\norm{\eta_k(\xi)\ft(u)}_{X_k}^2<\infty\},\\
&& N^{s}=  \{u\in \mathcal {S}'(\R^2): \norm{u}_{N^s}^2=\sum_{k\in
\Z_+}2^{2sk}\norm{(i+\tau-\xi^3)^{-1}\eta_k(\xi)\ft(u)}_{X_k}^2<\infty\}.
\end{eqnarray}

The space $F^s$ is between $X^{1/2,s}$ and $X^{1/2+,s}$. It can be
embedded into $C(\R;H^s)$ and into the  Strichartz-type space, say
$L^p_tL^q_x$ as $X^{1/2+,s}$. On the other hand, it has the same
scaling in time as $X^{1/2,s}$, which is crucial in the uniform
linear estimate, See section 4. That is the main reason for us
applying $F^s$.

For $T\geq 0$, we define the time-localized spaces
$X_T^{b,s,\alpha}$, $X_T^{b,s}$, $F^{s}(T)$, and $N^{s}(T)$
\begin{eqnarray}
&& \norm{u}_{X_T^{b,s,\alpha}}=\inf_{w\in
X^{b,s,\alpha}}\{\norm{w}_{X^{b,s,\alpha}},\  w(t)=u(t) \mbox{ on }
[0, T]\};\nonumber \\
&&\norm{u}_{X_T^{b,s}}=\inf_{w\in X^{b,s}}\{\norm{w}_{X^{b,s}}, \ \
w(t)=u(t) \mbox{ on } [0, T]\};\nonumber\\
&&\norm{u}_{F^{s}(T)}=\inf_{w\in F^{s}}\{\norm{w}_{F^{s}}, \
w(t)=u(t)
\mbox{ on } [0, T]\};\nonumber\\
&&\norm{u}_{N^{s}(T)}=\inf_{w\in N^{s}}\{\norm{w}_{N^{s}}, \
w(t)=u(t) \mbox{ on } [0, T]\}.
\end{eqnarray}

As a conclusion of this section we prove that the norm on $F^s$
controls some space-time norm as the norm $X^{1/2+,s}$. If applying
to frequency dyadic localized function, we see that the norm $F^s$
is almost the same as the norm $X^{1/2+,s}$. Fortunately, in
application we usually encounter this case. See \cite{Tao3} for a
survey on $X^{s,b}$ space.
\begin{proposition}\label{p21}
Let $Y$ be a Banach space of functions on $\R\times \R$ with the
property that
\[\norm{e^{it\tau_0}e^{-t\partial_x^3}f}_Y\les \ \norm{f}_{H^s(\R)}\]
holds for all $f\in H^s(\R)$ and $\tau_0\in \R$. Then we have the
embedding
\begin{equation}
\left(\sum_{k\in \Z_+}\norm{P_k u}_{Y}^2 \right)^{1/2}\lesssim
\norm{u}_{F^s}.
\end{equation}
\end{proposition}
\begin{proof}
In view of definition, it suffices to prove that if $k\in \Z_+$
\begin{equation}
\norm{P_ku}_Y \lesssim  2^{sk}\norm{\eta_k(\xi)\ft(u)}_{X_k}.
\end{equation}
Indeed, we have
\begin{align}
P_ku =&\int \eta_k(\xi)\ft u(\xi,\tau) e^{ix\xi}e^{it\tau}d\xi d\tau \nonumber\\
=&\sum_{j=0}^\infty\int \eta_j(\tau-\xi^3)\eta_k(\xi)\ft
u(\xi,\tau) e^{ix\xi}e^{it\tau}d\xi d\tau \nonumber\\
=&\sum_{j=0}^\infty\int \eta_j(\tau)e^{it\tau} \int \eta_k(\xi)\ft
u(\xi,\tau+\xi^3) e^{ix\xi}e^{it\xi^3}d\xi d\tau.
\end{align}
From the hypothesis on $Y$, we obtain
\begin{eqnarray}
\norm{P_k u}_Y &\les& \sum_{j=0}^\infty \int \eta_j(\tau)
\normo{e^{it\tau} \int \eta_k(\xi)\ft u(\xi,\tau+\xi^3)
e^{ix\xi}e^{it\xi^3}d\xi}_{Y}d\tau \nonumber\\
&\les& 2^{sk}\norm{\eta_k(\xi)\ft(u)}_{X_k},
\end{eqnarray}
which completes the proof of the proposition.
\end{proof}

\section{Global well-posedness for KdV-B equation}
In this section, we prove a global wellposedness result for the
KdV-Burgers equation by following the idea of Molinet and Ribaud
\cite{MR}. Using Duhamel's principle, we will mainly work on the
integral formulation of the KdV-Burgers equation
\begin{equation}\label{eq:intekdvb}
u(t)=W_\epsilon^\alpha(t)\phi_1-\half{1}\int_0^tW_\epsilon^\alpha(t-\tau)\partial_x(u^2(\tau))d\tau,
\ t\geq 0.
\end{equation}
We will apply a fixed point argument to solve the following
truncated version
\begin{equation}\label{eq:trunintekdvb}
u(t)=\psi(t)\left[W_\epsilon^\alpha(t)\phi_1-\half{\chi_{\R_+}(t)}\int_0^tW_\epsilon^\alpha(t-\tau)\partial_x(\psi_T^2(\tau)u^2(\tau))d\tau
\right],
\end{equation}
where $t\in \R$ and $\psi$ is a smooth time cutoff function
satisfying
\begin{equation}\label{eq:cutoff}
\psi \in C_0^\infty(\R), \quad \supp\ \psi \subset [-2, 2], \quad
\psi\equiv 1 \mbox{ on } [-1, 1],
\end{equation}
and $\psi_T(\cdot)=\psi(\cdot/T)$. Indeed, if $u$ solves
\eqref{eq:trunintekdvb} then $u$ is a solution of
\eqref{eq:intekdvb} on $[0, T]$, $T\leq 1$.

Theorem \ref{t11} can be proved by a slightly modified argument in
\cite{MR} combined with the following bilinear estimate. See also
\cite{Vento}.
\begin{proposition}\label{p31}
Let $s_\alpha$ be given by \eqref{eq:sa}. Let $s \in (s_\alpha, 0]$,
$0<\delta\ll 1$, then there exists $C_{s,\alpha}>0$ such that for
any $u, v \in \Sch$,
\begin{equation}\label{eq:bikdvb}
\norm{\partial_x(uv)}_{X^{-1/2+\delta, s, \alpha}}\leq
C_{s,\alpha}\norm{u}_{X^{1/2, s, \alpha}}\norm{v}_{X^{1/2, s,
\alpha}}.
\end{equation}
\end{proposition}

This type of estimate was systematically studied in \cite{Tao}, see
also \cite{KPV} for an elementary method. We will follow the idea in
\cite{Tao} to prove Proposition \ref{p31}. Let $Z$ be any abelian
additive group with an invariant measure $d\xi$. In particular,
$Z=\R^2$ in this paper. For any $k\geq 2$, Let $\Gamma_k(Z)$ denote
the hyperplane in $\R^k$
\[\Gamma_k(Z):=\{(\xi_1,\ldots, \xi_k)\in Z^k: \xi_1+\ldots+\xi_k=0\}\]
endowed with the induced measure
\[\int_{\Gamma_k(Z)}f:=\int_{Z^{k-1}}f(\xi_1,\ldots, \xi_{k-1},-\xi_1-\ldots-\xi_{k-1})d\xi_1\ldots d\xi_{k-1}.\]
Note that this measure is symmetric with respect to permutation of
the co-ordinates.

A function $m: \Gamma_k(Z)\rightarrow \C$ is said to to be a
$[k;Z]-multiplier$, and we define the norm $\norm{m}_{[k;Z]}$ to be
the best constant such that the inequality
\begin{equation}
\left|\int_{\Gamma_k(Z)}m(\xi)\prod_{j=1}^k f_i(\xi_i)\right|\leq
\norm{m}_{[k;Z]}\prod_{j=1}^k \norm{f_i}_{L^2}
\end{equation}
holds for all test functions $f_i$ on $Z$.

By duality and Plancherel's equality, it is easy to see that for
\eqref{eq:bikdvb}, it suffices to prove
\begin{equation}\label{eq:bi}
\normo{\frac{|\xi_3|\jb{\xi_3}^s\jb{\xi_1}^{-s}\jb{\xi_2}^{-s}\jb{i(\tau_3-\xi_3)+|\xi_3|^{2\alpha}}^{-1/2+\delta}}{\jb{i(\tau_2-\xi_2)+|\xi_2|^{2\alpha}}^{1/2}\jb{i(\tau_1-\xi_1)+|\xi_1|^{2\alpha}}^{1/2}}}_{[3;\R^2]}\les
1.
\end{equation}
By comparision principle (see \cite{Tao}), it suffices to prove that
\begin{eqnarray}
&&\sum_{N_1,N_2,N_3}\sum_{L_1,L_2,L_3}\sum_{H}\frac{N_3\jb{N_3}^s\jb{N_1}^{-s}\jb{N_2}^{-s}}{\jb{L_1+N_1^{2\alpha}}^{1/2}\jb{L_2+N_2^{2\alpha}}^{1/2}\jb{L_3+N_3^{2\alpha}}^{1/2-\delta}}\nonumber\\
&&\qquad
\qquad\norm{\chi_{N_1,N_2,N_3;H;L_1,L_2,L_3}}_{[3;\R^2]}\les 1,
\end{eqnarray}
where $N_i,L_i, H$ are dyadic, $h(\xi)=\xi_1^3+\xi_2^3+\xi_3^3$ and
\begin{eqnarray}
&&\chi_{N_1,N_2,N_3;H;L_1,L_2,L_3}=\chi_{|\xi_1|\sim N_1,|\xi_2|\sim
N_2,|\xi_3|\sim N_3}\nonumber\\
&&\qquad \cdot \chi_{|h(\xi)|\sim H}\chi_{|\tau_1-\xi_1^3|\sim
L_1,|\tau_2-\xi_2^3|\sim L_2,|\tau_3-\xi_3^3|\sim L_3}.
\end{eqnarray}
The issues reduce to an estimate of
\begin{equation}\label{eq:char}
\norm{\chi_{N_1,N_2,N_3;H;L_1,L_2,L_3}}_{[3;\R^2]}
\end{equation}
and dyadic summation. Since
\[\xi_1+\xi_2+\xi_3=0,\quad |h(\xi)|=|\xi_1^3+\xi_2^3+\xi_3^3|\sim N_1N_2N_3,\]
and
\[\tau_1-\xi_1^3+\tau_2-\xi_2^3+\tau_3-\xi_3^3+h(\xi)=0,\]
then we have
\begin{eqnarray}\label{eq:dyadic1}
N_{max}&\sim& N_{med},\nonumber\\
L_{max}&\sim& \max(L_{med}, H),
\end{eqnarray}
where we define $N_{max}\geq N_{med}\geq N_{min}$ to be the maximum,
median, and minimum of $N_1,\ N_2,\ N_3$ respectively. Similarly
define $L_{max}\geq L_{med}\geq L_{min}$. It's known (see Section 4,
\cite{Tao}) that we may assume
\begin{equation}\label{eq:dyadic2}
N_{max}\ges 1, \quad L_1,L_2,L_3\ges 1.
\end{equation}
Therefore, from Schur's test (Lemma 3.11, \cite{Tao}) it suffices to
prove that
\begin{eqnarray}\label{eq:bicase1}
&&\sum_{N_{max}\sim N_{med}\sim N}\sum_{L_1,L_2,L_3\geq 1}\frac{N_3\jb{N_3}^s\jb{N_1}^{-s}\jb{N_2}^{-s}}{\jb{L_1+N_1^{2\alpha}}^{1/2}\jb{L_2+N_2^{2\alpha}}^{1/2}\jb{L_3+N_3^{2\alpha}}^{1/2-\delta}}\nonumber\\
&&\qquad \qquad \times
\norm{\chi_{N_1,N_2,N_3;L_{max};L_1,L_2,L_3}}_{[3;\R^2]}
\end{eqnarray}
and
\begin{eqnarray}\label{eq:bicase2}
&&\sum_{N_{max}\sim N_{med}\sim N}\sum_{L_{max}\sim L_{med}}\sum_{H \leq L_{max}}\frac{N_3\jb{N_3}^s\jb{N_1}^{-s}\jb{N_2}^{-s}}{\jb{L_1+N_1^{2\alpha}}^{1/2}\jb{L_2+N_2^{2\alpha}}^{1/2}\jb{L_3+N_3^{2\alpha}}^{1/2-\delta}}\nonumber\\
&&\qquad \qquad \times
\norm{\chi_{N_1,N_2,N_3;H;L_1,L_2,L_3}}_{[3;\R^2]}
\end{eqnarray}
are both uniformly bounded for all $N\ges 1$.

\begin{proposition}[Proposition 6.1, \cite{Tao}]\label{pchar}
Let dyadic numbers $H,N_1,N_2,N_3,L_1,L_2,L_3>0$ obey \eqref{eq:dyadic1}, \eqref{eq:dyadic2}.\\
(i) If $N_{max}\sim N_{min}$ and $L_{max}\sim H$, then we have
\begin{equation}\label{eq:chari}
\eqref{eq:char} \les L_{min}^{1/2}N_{max}^{-1/4}L_{med}^{1/4}.
\end{equation}
(ii) If $N_2\sim N_3 \gg N_1$ and $H\sim L_1\ges L_2,L_3$, then
\begin{equation}\label{eq:charii}
\eqref{eq:char} \les
L_{min}^{1/2}N_{max}^{-1}\min(H,\frac{N_{max}}{N_{min}}L_{med})^{1/2}.
\end{equation}
Similarly for permutations.\\
(iii) In all other cases, we have
\begin{equation}\label{eq:chariii}
\eqref{eq:char} \les L_{min}^{1/2}N_{max}^{-1}\min(H,L_{med})^{1/2}.
\end{equation}
\end{proposition}

In order to estimate the denominator in \eqref{eq:bicase1},
\eqref{eq:bicase2}, we will need the following proposition to reduce
some cases.
\begin{proposition}\label{pdeno}
Let $k\in \N$. Assume that $a_1, a_2,\ldots, a_k$ and
$b_1,b_2,\ldots, b_k$ are non-negative numbers, and $A_1\leq A_2\leq
\ldots \leq A_k$, $B_1\leq B_2\leq \ldots \leq B_k$ are rearrange of
$\{a_i\}$, $\{b_i\}$ respectively. Then
\begin{equation}\label{eq:edeno}
\prod_{i=1}^k(a_i+b_i)\geq \prod_{i=1}^k(A_i+B_i).
\end{equation}
\end{proposition}
\begin{proof}
We apply an induction on $k$. The case $k=1$ is obviously. For
$k=2$, we have
\begin{eqnarray*}
(a_1+b_1)(a_2+b_2)&=&a_1a_2+b_1b_2+a_1b_2+a_2b_1\\
&\geq& A_1B_1+A_2B_2+A_1B_2+A_2B_1=(A_1+B_1)(A_2+B_2).
\end{eqnarray*}
We assume the lemma holds for all $ q\in \N,\ q\leq k-1$. Now we
prove for $k$. If $a_1=A_1$, $b_1=B_1$, then we apply induction
assumption for $k-1$ and get \eqref{eq:edeno}. Otherwise, we may
assume $a_1=A_1$, $b_2=B_1$. By induction assumption for $2$, then
$k-1$, we get
\begin{eqnarray}
\prod_{i=1}^k(a_i+b_i)&=&(a_1+b_1)(a_2+b_2)\prod_{i=3}^k(a_i+b_i)\nonumber\\
&\geq& (A_1+B_1)(a_2+b_1)\prod_{i=3}^k(a_i+b_i)\nonumber\\
&\geq&\prod_{i=1}^k(A_i+B_i),
\end{eqnarray}
which completes the proof of the proposition.
\end{proof}

\begin{proof}[Proof of Proposition \ref{p31}.]
We will prove the proposition using case-by-case analysis. We first
bound \eqref{eq:bicase2}. Since we have
\begin{equation}
N_3\jb{N_3}^s\jb{N_1}^{-s}\jb{N_2}^{-s}\les
N\jb{N_{min}}^{-s}+N^{-2s}N_{min}\jb{N_{min}}^s
\end{equation}
and from (iii) of Proposition \ref{pchar}, we obtain
\begin{eqnarray}
\eqref{eq:bicase2}&\les&\sum_{N_{max}\sim N_{med}\sim
N}\sum_{L_i,L_{max}\geq
H}\frac{(N\jb{N_{min}}^{-s}+N^{-2s}N_{min}\jb{N_{min}}^s)L_{min}^{1/2}N_{min}^{1/2}}{L_{max}^{1/2-\delta}L_{med}^{1/2-\delta}L_{min}^{1/2-\delta}}\nonumber
\\
&\les&\sum_{N_{max}\sim N_{med}\sim N}\sum_{L_{max}\geq
H}(N\jb{N_{min}}^{-s}+N^{-2s}N_{min}\jb{N_{min}}^s)
L_{max}^{-1+3\delta}N_{min}^{1/2}\nonumber\\
&\les&\sum_{N_{min}\leq N^{-2}}(N+N^{-2s}N_{min})
N_{min}^{1/2}\nonumber\\
&&+\sum_{N^{-2}\leq N_{min}\leq 1}(N+N^{-2s}N_{min})
N^{-2+6\delta}N_{min}^{-1/2+3\delta}\nonumber\\
&&+\sum_{N_{min}\geq 1}(NN_{min}^{-s}+N^{-2s}N_{min}^{1+s})
N^{-2+6\delta}N_{min}^{-1/2+3\delta}\nonumber\\
&\les&1,
\end{eqnarray}
provided that $-1<s\leq 0$.

We next bound \eqref{eq:bicase1}, which is more complicated. We
first assume that \eqref{eq:chari} applies. Then we have
\begin{eqnarray}
\eqref{eq:bicase1}&\les&\sum_{N_{max}\sim N_{min}\sim
N}\sum_{L_1,L_2,L_3\geq
1}\frac{N^{3/4-s}L_{min}^{1/2}L_{med}^{1/4}\jb{L_{min}+N^{2\alpha}}^{-1/2+\delta}}{\jb{L_{max}+N^{2\alpha}}^{1/2-\delta}\jb{L_{med}+N^{2\alpha}}^{1/2-\delta}}\nonumber\\
&\les&\sum_{N_{max}\sim N_{min}\sim N}\sum_{L_{med}}\frac{N^{3/4-s}L_{med}^{1/4+\delta}}{N^{3/2-3\delta}\jb{L_{med}+N^{2\alpha}}^{1/2-\delta}}\nonumber\\
&\les&N^{-\frac{3}{4}-\half \alpha-s+4\delta}\les 1,
\end{eqnarray}
provided that $-\frac{3}{4}-\half \alpha<s\leq 0$.

If \eqref{eq:chariii} applies, from Proposition \ref{pdeno}, we
obtain
\begin{eqnarray}
\eqref{eq:bicase1}&\les&\sum_{N_i}\sum_{L_i}\frac{(N\jb{N_{min}}^{-s}+N^{-2s}N_{min}\jb{N_{min}}^s)L_{min}^{1/2}N^{-1}L_{med}^{1/2}}{(L_{max}+N^{2\alpha})^{1/2-\delta}\jb{L_{med}+N^{2\alpha}}^{1/2-\delta}\jb{L_{min}+N_{min}^{2\alpha}}^{1/2-\delta}}\nonumber\\
&\les&\sum_{N_i}\frac{(N\jb{N_{min}}^{-s}+N^{-2s}N_{min}\jb{N_{min}}^s)N^{-1+4\alpha\delta}}{(N^2N_{min}+N^{2\alpha})^{1/2-3\delta}}\nonumber\\
&\les&\sum_{N_{min}\leq N^{2\alpha-2}}\frac{(N+N^{-2s}N_{min})N^{-1+4\alpha\delta}}{N^{\alpha-6\delta}}\nonumber\\
&&+\sum_{N^{2\alpha-2}\leq N_{min}\leq 1}\frac{(N+N^{-2s}N_{min})N^{-1+4\alpha\delta}}{N^{1-6\delta}N_{min}^{1/2-3\delta}}\nonumber\\
&&+\sum_{N_{min}\geq
1}\frac{(NN_{min}^{-s}+N^{-2s}N_{min}^{1+s})N^{-1+4\alpha\delta}}{N^{1-6\delta}N_{min}^{1/2-3\delta}}\nonumber\\
&\les&N^{-\alpha+10\delta}+N^{-2s-3+\alpha+6\delta}+N^{-2s-2+6\delta}+N^{-s-3/2+7\delta}\nonumber\\
&\les&1,
\end{eqnarray}
provided that $-1<s\leq 0$.

If \eqref{eq:charii} applies, we have three cases:
\begin{eqnarray}
&&N_2\sim N_3\gg N_1,\quad L_1\ges L_2, L_3,\label{eq:bicase121}\\
&&N_1\sim N_3\gg N_2,\quad L_2\ges L_1, L_3,\label{eq:bicase122}\\
&&N_1\sim N_2\gg N_3,\quad L_3\ges L_1, L_2.\label{eq:bicase123}
\end{eqnarray}
If \eqref{eq:bicase121} holds, then we have
\begin{eqnarray}
\eqref{eq:bicase1}&\les&\sum_{N_i}\sum_{L_i}\frac{N\jb{N_{min}}^{-s}L_{min}^{1/2}N^{-1}\min(H,\frac{N_{max}}{N_{min}}L_{med})^{1/2}}{N_{min}^{1/2}N\jb{L_{med}+N^{2\alpha}}^{1/2-\delta}\jb{L_{min}+N^{2\alpha}}^{1/2}}\nonumber\\
&\les&\sum_{N_i}\sum_{L_{med}\geq
NN_{min}^2}\frac{N\jb{N_{min}}^{-s}\log(L_{med})N^{-1}N_{min}^{1/2}N}{N_{min}^{1/2}N\jb{L_{med}+N^{2\alpha}}^{1/2-\delta}}\nonumber\\
&&+\sum_{N_i}\sum_{L_{med}\leq
NN_{min}^2}\frac{N\jb{N_{min}}^{-s}\log(L_{med})L_{med}^{1/2}N^{-1}N_{min}^{-1/2}N^{1/2}}{N_{min}^{1/2}N\jb{L_{med}+N^{2\alpha}}^{1/2-\delta}}\nonumber\\
&=&A_1+A_2.
\end{eqnarray}
We first bound $A_1$.
\begin{eqnarray}
A_1&\les&\sum_{N^{-2}\leq N_{min}\leq 1}\sum_{L_{med}\geq
NN_{min}^2}\frac{L_{med}^\delta}{\jb{L_{med}+N^{2\alpha}}^{1/2-\delta}}\nonumber\\
&&+\sum_{N_{min}\geq 1}\sum_{L_{med}\geq NN_{min}^2}\frac{L_{med}^\delta N_{min}^{-s}}{\jb{L_{med}}^{1/2-\delta}}\nonumber\\
&\les& N^{-\alpha+7\delta}+\sum_{N_{min}\geq
1}N_{min}^{-s-1+4\delta}N^{-1/2+2\delta}\les 1,
\end{eqnarray}
provided $-1<s\leq 0$.

For $A_2$, we have
\begin{eqnarray}
A_2&\les&\sum_{N^{-1/2}\leq N_{min}\leq 1}\sum_{L_{med}\leq
NN_{min}^2}\frac{L_{med}^{\delta+1/2}N_{min}^{-1}N^{-1/2}}{\jb{L_{med}+N^{2\alpha}}^{1/2-\delta}}\nonumber\\
&&+\sum_{N_{min}\geq 1}\sum_{L_{med}\leq
NN_{min}^2}\frac{L_{med}^{\delta+1/2}N_{min}^{-1-s}N^{-1/2}}{\jb{L_{med}+N^{2\alpha}}^{1/2-\delta}}\nonumber\\
&\les&\sum_{N^{-1/2}\leq N_{min}\leq
1}N^{2\delta-1/2}N_{min}^{4\delta-1}+\sum_{N_{min}\geq
1}N_{min}^{-1-s+4\delta}N^{-1/2+2\delta}\nonumber\\
&\les&1,
\end{eqnarray}
provided $-1<s\leq 0$.

 From symmetry, the case
\eqref{eq:bicase121} is identical to the case \eqref{eq:bicase122}.
Now we assume that  \eqref{eq:bicase123} holds, and we obtain
\begin{eqnarray}
\eqref{eq:bicase1}&\les&\sum_{N_i}\sum_{L_i}\frac{N^{-2s}\jb{N_{min}}^{s}N_{min}L_{min}^{1/2}N^{-1}\min(H,\frac{N_{max}}{N_{min}}L_{med})^{1/2}}{N_{min}^{1/2-\delta}N^{1-2\delta}\jb{L_{med}+N^{2\alpha}}^{1/2-\delta}\jb{L_{min}+N^{2\alpha}}^{1/2}}\nonumber\\
&\les&\sum_{N_i}\sum_{L_{med}\geq
NN_{min}^2}\frac{N^{-2s}\jb{N_{min}}^{s}N_{min}\log(L_{med})N^{-1}N_{min}^{1/2}N}{N_{min}^{1/2-\delta}N^{1-2\delta}\jb{L_{med}+N^{2\alpha}}^{1/2-\delta}}\nonumber\\
&&+\sum_{N_i}\sum_{L_{med}\leq
NN_{min}^2}\frac{N^{-2s}\jb{N_{min}}^{s}N_{min}\log(L_{med})L_{med}^{1/2}N^{-1}N_{min}^{-1/2}N^{1/2}}{N_{min}^{1/2-\delta}N^{1-2\delta}\jb{L_{med}+N^{2\alpha}}^{1/2}}\nonumber\\
&=&B_1+B_2.
\end{eqnarray}
We first bound $B_1$.
\begin{eqnarray}
B_1&\les&\sum_{N^{-2}\leq N_{min}\leq 1}\sum_{L_{med}\geq
NN_{min}^2}\frac{N^{-2s-1+2\delta}N_{min}^{1+\delta}L_{med}^\delta}{\jb{L_{med}+N^{2\alpha}}^{1/2-\delta}}\nonumber\\
&&+\sum_{N_{min}\geq 1}\sum_{L_{med}\geq
NN_{min}^2}\frac{N^{-2s-1+2\delta}N_{min}^{1+\delta+s}L_{med}^\delta}{\jb{L_{med}+N^{2\alpha}}^{1/2-\delta}}\nonumber\\
&\les&\sum_{N^{-2}\leq N_{min}\leq 1}\frac{N^{-2s-1+2\delta}N_{min}^{1+\delta}}{\jb{NN_{min}^2+N^{2\alpha}}^{1/2-2\delta}}\nonumber\\
&&+\sum_{N_{min}\geq
1}\frac{N^{-2s-1+2\delta}N_{min}^{1+\delta+s}}{\jb{NN_{min}^2+N^{2\alpha}}^{1/2-2\delta}}.
\end{eqnarray}
We discuss it in the following two cases. If $1/2\leq \alpha\leq 1$,
then
\begin{eqnarray}
B_1&\les&N^{-2s-1-\alpha+6\delta} +\sum_{N_{min}\geq
N^{\alpha-1/2}}N^{-2s-3/2+4\delta}N_{min}^{5\delta+s}\nonumber\\
&&+\sum_{1\leq N_{min}\leq
N^{\alpha-1/2}}N^{-2s-1-\alpha+6\delta}N_{min}^{1+\delta+s},
\end{eqnarray}
provided that $-\frac{3}{5-2\alpha}<s\leq 0$. If $0<\alpha \leq
1/2$, then
\begin{eqnarray}
B_1&\les&\sum_{N^{\alpha-1/2}\leq N_{min}\leq
1}N^{-2s-3/2+4\delta}N_{min}^{5\delta}+\sum_{N_{min}\geq
1}N^{-2s-3/2+4\delta}N_{min}^{5\delta+s}\nonumber\\
&&+\sum_{N^{-2}\leq N_{min}\leq
N^{\alpha-1/2}}N^{-2s-1-\alpha+6\delta}N_{min}^{1+\delta}\nonumber\\
&\les&1,
\end{eqnarray}
provided that $-3/4<s\leq 0$.

For $B_2$, we have
\begin{eqnarray}
B_2&\les&\sum_{N^{-1/2}\leq N_{min}\leq 1}\sum_{L_{med}\leq
NN_{min}^2}\frac{N^{-2s-3/2+2\delta}N_{min}^\delta L_{med}^{1/2+\delta}}{\jb{L_{med}+N^{2\alpha}}^{1/2}}\nonumber\\
&&+\sum_{N_{min}\geq 1}\sum_{L_{med}\leq
NN_{min}^2}\frac{N^{-2s-3/2+2\delta}N_{min}^{\delta+s}
L_{med}^{1/2+\delta}}{\jb{L_{med}+N^{2\alpha}}^{1/2}}\nonumber.
\end{eqnarray}
and get
\begin{eqnarray}
B_2&\les&\sum_{N^{-1/2}\leq N_{min}\leq
1}\frac{N^{-2s-1+3\delta}N_{min}^{1+3\delta}}{\jb{NN_{min}^2+N^{2\alpha}}^{1/2}}+\sum_{N_{min}\geq
1}\frac{N^{-2s-1+3\delta}N_{min}^{1+s+3\delta}}{\jb{NN_{min}^2+N^{2\alpha}}^{1/2}}\nonumber.
\end{eqnarray}
If $1/2\leq \alpha \leq 1$, then
\begin{eqnarray}
B_2&\les&N^{-2s-1-\alpha+3\delta}+\sum_{N_{min}\geq
N^{\alpha-1/2}}N^{-2s-3/2+3\delta}N_{min}^{s+3\delta}\nonumber\\
&&+\sum_{1\leq N_{min}\leq
N^{\alpha-1/2}}N^{-2s-1-\alpha+3\delta}N_{min}^{1+s+3\delta}\nonumber\\
&\les&1,
\end{eqnarray}
provided that $-\frac{3}{5-2\alpha}<s\leq 0$. If $0< \alpha \leq
1/2$, then
\begin{eqnarray}
B_2&\les&\sum_{N^{-1/2}\leq N_{min}\leq
N^{\alpha-1/2}}N^{-2s-1-\alpha+3\delta}N_{min}^{1+3\delta}\nonumber\\
&&+\sum_{N^{\alpha-1/2}\leq N_{min}\leq
1}N^{-2s-3/2+3\delta}N_{min}^{3\delta}+\sum_{N_{min}\geq
1}N^{-2s-3/2+3\delta}N_{min}^{s+3\delta}\nonumber \\
&\les&1,
\end{eqnarray}
provided that $-3/4<s\leq 0$. Therefore, we complete the proof of
Proposition \ref{p31}.
\end{proof}

\begin{proposition}
If $s\leq s_\alpha$, then for any $0<\delta\ll 1$, there doesn't
exist $C>0$ such that for any $u, v \in \Sch$,
\begin{equation}\label{eq:bikdvb}
\norm{\partial_x(uv)}_{X^{-1/2+\delta, s, \alpha}}\leq
C\norm{u}_{X^{1/2, s, \alpha}}\norm{v}_{X^{1/2, s, \alpha}}.
\end{equation}
\end{proposition}
\begin{proof}
From the proof of the Proposition \ref{p31}, we see that the
restriction on $s$ is caused by high-high interaction, and hence we
construct the worst case. The idea is due to C. Kenig, G. Ponce and
L. Vega \cite{KPV}. In view of definition, \eqref{eq:bikdvb} is
equivalent to
\begin{eqnarray}\label{eq:count}
&& \norm{\frac{\xi(1+|\xi|)^s}{(1+|\xi|^{2\alpha}+|\tau-\xi^3|)^{1/2-\delta}}\nonumber\\
&&\times \int
\frac{f(\xi_1,\tau_1)(1+|\xi_1|)^{-s}f(\xi-\xi_1,\tau-\tau_1)(1+|\xi-\xi_1|)^{-s}d\xi_1
d\tau_1}{\jb{|\xi_1|^{2\alpha}+|\tau_1-\xi_1^3|}^{1/2}\jb{|\xi-\xi_1|^{2\alpha}+|\tau-\tau_1-(\xi-\xi_1)^3|}^{1/2}}}_{L_{\xi,\tau}^2}\nonumber\\
&\les& \norm{f}_{L_{\xi,\tau}^2}^2.
\end{eqnarray}
If $0<\alpha\leq 1/2$, fix $N\gg 1$, we set
\begin{eqnarray*}
f(\xi,\tau)=\chi_A(\xi,\tau)+\chi_{-A}(\xi,\tau),
\end{eqnarray*}
where
\begin{eqnarray*}
A=\{(\xi,\tau)\in \R^2|N\leq \xi \leq N+1, N\leq |\tau-\xi^3|\leq
2N\},
\end{eqnarray*}
and
\begin{eqnarray*}
-A=\{(\xi,\tau)\in \R^2| -(\xi,\tau)\in A\}.
\end{eqnarray*}
Clearly,
\begin{equation}
\norm{f}_{L_{\xi,\tau}^2}\sim N^{1/2}.
\end{equation}
On the other hand, $A$ contains a rectangle with $(N, N^3+N)$ as a
vertex, with dimension $N^{-1}\times N^2$ and longest side pointing
in the $(1, 3N^2)$ direction. Therefore,
\begin{equation}
|f*f(\xi,\tau)|\ges N\chi_{R}(\xi,\tau),
\end{equation}
where R is a rectangle centered at the origin of dimensions
$N^{-1}\times N^2$ and longest side pointing in the $(1, 3N^2)$
direction. Taking the one-third rectangle away from origin, then we
have $|\xi|\sim 1$, and therefore \eqref{eq:count} implies that
\begin{equation}
N^{-1+2\delta}N^{-2s}N^{-1}NN^{-1/2}N\les N,
\end{equation}
which implies that $s> -3/4$.

If $1/2\leq \alpha \leq 1$, then take
\begin{eqnarray*}
f(\xi,\tau)=\chi_B(\xi,\tau)+\chi_{-B}(\xi,\tau),
\end{eqnarray*}
where
\begin{equation}
B=\{(\xi,\tau)\in \R^2|N\leq \xi \leq N+N^{\alpha-1/2},
N^{2\alpha}\leq |\tau-\xi^3|\leq 2N^{2\alpha}\},
\end{equation}
and
\begin{eqnarray*}
-B=\{(\xi,\tau)\in \R^2| -(\xi,\tau)\in B\}.
\end{eqnarray*}
Clearly,
\begin{equation}
\norm{f}_{L_{\xi,\tau}^2}\sim N^{\half {3\alpha}-\frac{1}{4}}.
\end{equation}
On the other hand, $B$ contains a rectangle with $(N,
N^3+N^{2\alpha})$ as a vertex, with dimension $N^{2\alpha-2}\times
N^{\alpha+3/2}$ and longest side pointing in the $(1, 3N^2)$
direction. Therefore,
\begin{equation}
|f*f(\xi,\tau)|\ges N^{3\alpha-1/2}\chi_{R}(\xi,\tau),
\end{equation}
where R is a rectangle centered at the origin of dimensions
$N^{2\alpha-2}\times N^{\alpha+3/2}$ and longest side pointing in
the $(1, 3N^2)$ direction. Taking the one-third rectangle away from
origin, then we have $|\xi|\sim N^{\alpha-1/2}$, and therefore
\eqref{eq:count} implies that
\begin{equation}
N^{(\alpha-1/2)(1+s)}N^{(\alpha+3/2)(-1/2+\delta)}N^{-2s}N^{-2\alpha}N^{3\alpha-1/2}N^{\alpha-1}N^{\alpha/2+3/4}\les
 N^{3\alpha-1/2},
\end{equation}
which implies that $s> -3/(5-2\alpha)$.
\end{proof}

\begin{remark}\label{r35} \rm
The constant in Proposition \ref{p31} depends on $\alpha$, which is
the main reason for gaining $\delta$-order derivative in time in the
bilinear estimates. In proving global well-posedness we also need to
exploit the smoothing effect of the dissipative term and then $L^2$
conservation law. Therefore, the result of Theorem \ref{t11} is
dependent of $\epsilon$.
\end{remark}

\section{Uniform LWP for KdV-B equation}
In this section we study the uniform local well posedness for the
KdV-Burgers equation. We will prove a time localized version of
Theorem \ref{t12} where $T=T(\norm{\phi}_{H^{s}})$ is small. In view
of Remark \ref{r35}, the space $X^{b,s}$ we used in the last section
is not proper in this situation. We will use the space $F^s$. Let us
recall that \eqref{eq:kdvb} is invariant in the following scaling
\begin{equation}\label{eq:scaling}
u(x,t)\rightarrow \lambda^2 u(\lambda x, \lambda^3 t),\
\phi(x)\rightarrow \lambda^2\phi(\lambda x), \ \epsilon\rightarrow
\lambda^{3-2\alpha}\epsilon, \ \ \forall \  0<\lambda\leq 1.
\end{equation}
This invariance is very important in the proof of Theorem \ref{t12}
and also crucial for the uniform global-well posedness in the next
section. We first show that $F^s(T)\hookrightarrow C([0, T], H^{s})$
for $s\in \R$, $T\in (0, 1]$ in the following proposition.
\begin{proposition}\label{p41}
If $s\in \R$, $T\in (0, 1]$, and $u\in F^s(T)$, then
\begin{equation}
\sup_{t\in [0, T]}\norm{u(t)}_{H^s}\les \norm{u}_{F^s(T)}.
\end{equation}
\end{proposition}
\begin{proof}
In view of definition, it suffices to show that for $k\in \Z_+$,
$t\in [0, 1]$,
\begin{equation}\label{eq:s42}
\norm{\eta_k(\xi)\ft_x{u}(t)}_{L^2}\les
\norm{\eta_k(\xi)\ft{u}}_{X_k}.
\end{equation}
From the fact
\[
\eta_k(\xi)\ft_x{u}(t)= \sum_{j\in \mathbb{Z}_+}\int_\R
\eta_j(\tau-\xi^3)\eta_k(\xi)\ft(u)(\tau)e^{it\tau}d\tau,
\]
we easily see that \eqref{eq:s42} follows from the Minkowski's
inequality, Cauchy-Schwarz inequality and the definition of $X_k$.
\end{proof}

We prove an embedding property of the space $N^s$ in the next
proposition which can be viewed as a dual version of Proposition
\ref{p41}. This property is important in proving the limit behavior
in Section 6.
\begin{proposition}\label{p42}
If $s\in \R$ and $u\in L_t^2H_x^s$, then
\begin{equation}
\norm{u}_{N^s}\les \norm{u}_{L_t^2H_x^s}.
\end{equation}
\begin{proof}
We may assume $s=0$. By definition it suffices to prove that for
$k\in \Z_+$,
\begin{equation}
\norm{(i+\tau-\xi^3)^{-1}\eta_k(\xi)\ft(u)}_{X_k}\les
\norm{\eta_k(\xi)\ft(u)}_{L^2},
\end{equation}
which immediately follows from the definition of $X_k$.
\end{proof}
\end{proposition}

As in the last section we will mainly work on the correspondng
integral equation of eq. \eqref{eq:kdvb}. But for technical reason
we will mainly work on the following integral equation
\begin{equation}\label{eq:trunintekdvbuni}
u(t)=\psi(t)\left[W_\epsilon^\alpha(t)\phi_1-L \big(\partial_x
(\psi^2 u^2)\big)(x,t) \right],
\end{equation}
where $\psi$ is as in \eqref{eq:cutoff} and
\begin{eqnarray} \label{defL}
L(f)(x,t)=W_0(t)\int_{\R^2}e^{ix\xi}\frac{e^{it\tau'}-e^{-\epsilon
|t||\xi|^{2\alpha}}}{i\tau'+\epsilon
|\xi|^{2\alpha}}\ft(W_0(-t)f)(\xi,\tau')d\xi d\tau'.
\end{eqnarray}
One easily sees that
\begin{eqnarray}
\chi_{\R_+}(t) \psi(t)L(f)(x,t)&=&\chi_{\R_+}(t)\psi(t)\int_0^t
W_\epsilon^\alpha(t-\tau) f(\tau)d\tau. \label{eq:nonhomo}
\end{eqnarray}
Indeed, taking $w=W_0(\cdot)f$,  the right hand side of
\eqref{eq:nonhomo} can be rewritten as
\begin{eqnarray*}
&& W_0(t)\left[\chi_{\R_+}(t)\psi(t)\int_{\R^2}
e^{ix\xi}e^{-\epsilon t |\xi|^{2\alpha}}\widehat{w}(\xi,
\tau')\int_{0}^t e^{i\tau\tau'}e^{e^{\epsilon \tau
|\xi|^{2\alpha}}}d\tau d\xi
d\tau'\right]\\
&& =W_0(t)\left[\chi_{\R_+}(t)\psi(t)\int_{\R^2}
e^{ix\xi}\frac{e^{it\tau'}-e^{-\epsilon t
|\xi|^{2\alpha}}}{i\tau'+\epsilon|\xi|^{2\alpha}}\widehat{w}(\xi,
\tau') d\xi d\tau'\right].
\end{eqnarray*}
Thus, if $u$ solves \eqref{eq:trunintekdvbuni} then $u$ is a
solution of \eqref{eq:intekdvb} on $[0, 1]$. We first prove a
uniform estimate for the free solution.

\begin{proposition}\label{p43}
Let $s\in \R$. There exists $C>0$ such that for any $0\leq
\epsilon\leq 1$
\begin{equation}
\norm{\psi(t)W_\epsilon^\alpha(t)\phi}_{F^s}\leq C\norm{\phi}_{H^s},
\quad \forall \ \phi \in H^s(\R).
\end{equation}
\end{proposition}
\begin{proof}
We only prove the case $0<\epsilon\leq 1$. By definition of $F^s$,
it suffices to prove that for $k\in \Z_+$
\begin{equation}
\norm{\eta_k(\xi)\ft(\psi(t)W_\epsilon^\alpha(t)\phi)}_{X_k}\les
\norm{\eta_k(\xi)\widehat{\phi}(\xi)}_{L^2}. \label{eq:4.7}
\end{equation}
In view of the definition, if $k=0$, then by Taylor's expansion
\begin{eqnarray*}
&&\norm{\eta_0(\xi)\ft(\psi(t)W_\epsilon^\alpha(t)\phi)}_{X_0}\\
&\les&
\sum_{j=0}^{\infty}2^{j/2}\normo{\eta_0(\xi)\widehat{\phi}(\xi)\ft_t\brk{\psi(t)\sum_{n\geq 0}\frac{(-1)^n\epsilon^n |\xi|^{2n\alpha}}{n!}|t|^n}(\tau)\eta_j(\tau)}_{L_{\xi,\tau}^2}\\
&\les& \sum_{n\geq
0}\frac{4^n}{n!}\norm{\eta_0(\xi)\widehat{\phi}(\xi)}_{L^2}\norm{|t|^n\psi(t)}_{H^1}
\les\norm{\eta_0(\xi)\widehat{\phi}(\xi)}_{L^2},
\end{eqnarray*}
which is the estimate \eqref{eq:4.7}, as desired. We  now consider
the cases $k\geq 1$. We first observe that if $|\xi|\sim 2^k$, then
for any $ j\geq 0$,
\begin{equation}\label{eq:obser}
\norm{P_j(e^{-\epsilon |\xi|^{2\alpha}|t|})(t)}_{L^2}\les
\norm{P_j(e^{-\epsilon 2^{2k\alpha}|t|})(t)}_{L^2},
\end{equation}
which follows from Plancherel's equality and the fact that
\[\ft(e^{-|t|})(\tau)=C\rev{1+|\tau|^2}.\]
It follows from the definition that
\begin{eqnarray*}
\norm{\eta_k(\xi)\ft(\psi(t)W_\epsilon^\alpha(t)\phi)}_{X_k}&\les&\sum_{j=0}^\infty
2^{j/2}
\normo{\eta_k(\xi)\widehat{\phi}(\xi)\eta_j(\tau)\ft_t\brk{\psi(t)e^{-\epsilon
|t||\xi|^{2\alpha}}}(\tau)}_{L_{\xi,\tau}^2}\\
&\les&\sum_{j=0}^\infty 2^{j/2}
\normo{\eta_k(\xi)\widehat{\phi}(\xi)P_j\brk{\psi(t)e^{-\epsilon
|t||\xi|^{2\alpha}}}(t)}_{L_{\xi,t}^2}\\
&\les&\sum_{j=0}^\infty 2^{j/2}
\normo{\eta_k(\xi)\widehat{\phi}(\xi)}_{L^2} \sup_{|\xi|\sim
2^k}\normo{P_j\brk{\psi(t)e^{-\epsilon
|t||\xi|^{2\alpha}}}(t)}_{L_t^2}.
\end{eqnarray*}
It suffices to show that for any $k\geq 1$,
\begin{eqnarray}\label{eq:p442}
\sum_{j=0}^\infty 2^{j/2} \sup_{|\xi|\sim 2^k}
\normo{P_j\brk{\psi(t)e^{-\epsilon
|t||\xi|^{2\alpha}}}(t)}_{L_t^2}\les 1.
\end{eqnarray}
We may assume $j\geq 100$ in the summation. Using the para-product
decomposition, we have
\begin{align}
u_1u_2 = \sum^\infty_{r=0} [(P_{r+1} u_1) (P_{\le r+1}u_2)+ (P_{\le
r}u_1) (P_{r+1}u_2)], \label{eq:para-pr}
\end{align}
and
\begin{align}
P_j(u_1u_2) = P_j \Big(\sum_{r\ge j-10} [(P_{r+1} u_1) (P_{\le
r+1}u_2)+ (P_{\le r}u_1) (P_{r+1}u_2)]\Big):=P_j(I+II).
\label{eq:para-pr-2}
\end{align}
Now we take $u_1=\psi(t)$ and $u_2= e^{-\epsilon
|t||\xi|^{2\alpha}}$.   It follows from Bernstein's estimate,
H\"older's inequality and \eqref{eq:obser} that
\begin{align}
\sum_{j \ge 100} 2^{j/2} \|P_j(II)\|_{L^\infty_\xi L^2_t} & \lesssim
\sum_{j \ge 100} 2^{j/2}  \sum_{r\ge j-10} \| P_{r+1}
u_2\|_{L^{\infty}_\xi L^{2}_t}
\|P_{\le r+1}u_1 \|_{L^\infty_{\xi,t}} \nonumber\\
& \lesssim \sum_{j\ge 100} 2^{(j-r)/2} \sum_{r\ge j-10} 2^{r/2}  \|
P_{r+1} u_2\|_{L^{\infty}_\xi L^{2}_t}
 \nonumber\\
& \lesssim \sum_{r}  2^{r/2}  \| P_{r+1} (e^{-\epsilon
|t|2^{2k\alpha}})\|_{L^{2}_t} \les 1,
\end{align}
where we used the fact that $\dot{B}_{2,1}^{1/2}$ has a scaling
invariance and $e^{-|t|}\in\dot{B}_{2,1}^{1/2}$. the first  term
$P_j(I)$ in \eqref{eq:para-pr-2} can be handled in an easier way.
Therefore, we complete the proof of the proposition.
\end{proof}

From the proof we see that $F^s$ norm has a same scale in time as
$B^{1/2}_{2,1}$ and $e^{-\epsilon C |t|}$. If applying $X^{1/2+,s}$
norm, one can not get a uniform estimate. Similarly for the
inhomogeneous linear operator we get

\begin{proposition}\label{p45}
Let $s\in \R$. There exists $C>0$ such that for all $v \in
\Sch(\R^2)$ and  $0\leq \epsilon\leq 1$,
\begin{equation}
\norm{\psi(t)L(v)}_{F^s}\leq C \norm{v}_{N^s}.
\end{equation}
\end{proposition}
\begin{proof}
The idea is essential due to Molinet and Ribaud \cite{MR}. See also
section 5 in \cite{In-Ke}. We only prove the case $0<\epsilon\leq
1$. In view of definition, it suffices to prove that if $k\in \Z_+$,
\begin{eqnarray}
\norm{\eta_k(\xi)\ft(\psi(t)L(v))}_{X_k} \les
\norm{(i+\tau-\xi^3)^{-1}\eta_k(\xi)\ft(v)}_{X_k}.
\end{eqnarray}
We set \[w(\tau)=W_0(-\tau)v(\tau), \quad k_\xi(t)=\psi(t)\int_\R
\frac{e^{it\tau'}-e^{-\epsilon t
|\xi|^{2\alpha}}}{i\tau'+\epsilon|\xi|^{2\alpha}}\widehat{w}(\xi,
\tau') d\tau',\]
Therefore, by the definition, it suffices to prove
that
\begin{equation}\label{eq:p452}
\sum_{j=0}2^{j/2}\norm{\eta_k(\xi)\eta_j(\tau)\ft_t(k_\xi)(\tau)}_{L_{\xi,\tau}^2}
\les \sum_{j=0}2^{-j/2}\norm{\eta_k(\xi)\eta_j(\tau)\widehat{w}(\xi,
\tau)}_{L_{\xi,\tau}^2}.
\end{equation}
We first write
\begin{eqnarray*}
k_\xi(t)&=&\psi(t)\int_{|\tau|\leq
1}\frac{e^{it\tau}-1}{i\tau+\epsilon|\xi|^{2\alpha}}\widehat{w}(\xi,\tau)d\tau+\psi(t)\int_{|\tau|\leq
1}\frac{1-e^{-\epsilon
|t||\xi|^{2\alpha}}}{i\tau+\epsilon|\xi|^{2\alpha}}\widehat{w}(\xi,\tau)d\tau\\
&&+\psi(t)\int_{|\tau|\geq
1}\frac{e^{it\tau}}{i\tau+\epsilon|\xi|^{2\alpha}}\widehat{w}(\xi,\tau)d\tau-\psi(t)\int_{|\tau|\geq
1}\frac{e^{-\epsilon
|t||\xi|^{2\alpha}}}{i\tau+\epsilon|\xi|^{2\alpha}}\widehat{w}(\xi,\tau)d\tau\\
&=&I+II+III-IV.
\end{eqnarray*}
We now estimate the contributions of $I-IV$. First, we consider the
contribution of $IV$.
\begin{eqnarray*}
\sum_{j=0}2^{j/2}\norm{\eta_k(\xi)P_j(IV)(t)}_{L_{\xi,t}^2}&\leq&
\sum_{j=0}2^{j/2}\sup_{\xi\in I_k}\norm{\eta_k(\xi)P_j(\psi(t)e^{-\epsilon|t||\xi|^{2\alpha}})(t)}_{L_{t}^2}\nonumber\\
&&\quad \cdot\int_{|\tau|\geq 1}\frac{\norm{|\eta_k(\xi)\widehat{w}(\xi,\tau)|}_{L_\xi^2}}{|\tau|}d\tau\\
&\les&\sum_{j=0}2^{-j/2}\norm{\eta_k(\xi)\eta_j(\tau)\widehat{w}(\xi,
\tau)}_{L_{\xi,\tau}^2},
\end{eqnarray*}
where we use Taylor expansion for $k=0$ and \eqref{eq:p442} for
$k\geq 1$. Next, we consider the contribution of $III$. Setting
$g(\xi,\tau)=\frac{|\widehat{w}(\xi,\tau)|}{|i\tau+\epsilon|\xi|^{2\alpha}|}\chi_{|\tau|\geq
1}$ we have
\begin{eqnarray*}
\sum_{j=0}2^{j/2}\norm{\eta_k(\xi)P_j(III)(t)}_{L_{\xi,t}^2}&\les&\sum_{j=0}2^{j/2}\norm{\eta_k(\xi)\eta_j(\tau)\widehat{\psi}*_\tau g(\xi,\tau)}_{L_{\xi,\tau}^2}\\
&\les&\sum_{j\geq 1} 2^{j/2}\normo{
\frac{\eta_j(\tau')\norm{\eta_k(\xi)\widehat{w}(\xi,\tau')}_{L_\xi^2}}{|i\tau'|}\chi_{|\tau'|\geq
1}}_{L_{\tau'}^2}\nonumber\\
&\les&\sum_{j=0}2^{-j/2}\norm{\eta_k(\xi)\eta_j(\tau)\widehat{w}(\xi,
\tau)}_{L_{\xi,\tau}^2},
\end{eqnarray*}
where we used the fact that $B_{2,1}^{1/2}$ is a multiplication
algebra and that $\ft^{-1}(|\widehat{\psi}|)\in B_{2,1}^{1/2}$.
Thirdly, we consider the contribution of $II$. For
$\epsilon|\xi|^{2\alpha}\geq 1$, as for $IV$, we get
\begin{eqnarray*}
\sum_{j=0}2^{j/2}\norm{\eta_k(\xi)P_j(II)(t)}_{L_{\xi,t}^2}&\les&
\sum_{j=0}2^{j/2}\sup_{\xi\in I_k}\norm{\eta_k(\xi)P_j(\psi(1-e^{-\epsilon|t||\xi|^{2\alpha}}))(t)}_{L_{t}^2}\\
&&\cdot
\int\frac{\norm{\widehat{w}(\xi,\tau)}_{L_\xi^2}}{\jb{\tau}}d\tau
\nonumber\\
&\les&
\sum_{j=0}2^{-j/2}\norm{\eta_k(\xi)\eta_j(\tau)\widehat{w}(\xi,
\tau)}_{L_{\xi,\tau}^2}.
\end{eqnarray*}
For $\epsilon|\xi|^{2\alpha}\leq 1$, using Taylor's expansion, we
have
\begin{eqnarray*}
&&\sum_{j=0}2^{j/2}\norm{\eta_k(\xi)P_j(II)(t)}_{L_{\xi,t}^2}\\
&\les& \sum_{n\geq
1}\sum_{j=0}2^{j/2}\normo{\eta_k(\xi)\int_{|\tau|\leq
1}\frac{\widehat{w}(\xi,\tau)}{i\tau+\epsilon|\xi|^{2\alpha}}d\tau
P_j(|t|^n\psi(t))\frac{\epsilon^n|\xi|^{2\alpha
n}}{n!}}_{L_{\xi,t}^2} \\
&\les& \normo{\int_{|\tau|\leq
1}\frac{\epsilon|\xi|^{2\alpha}|\eta_k(\xi)\widehat{w}(\xi,\tau)|}{|i\tau+\epsilon|\xi|^{2\alpha}|}d\tau}_{L_\xi^2}\les\sum_{j=0}2^{-j/2}\norm{\eta_k(\xi)\eta_j(\tau)\widehat{w}(\xi,
\tau)}_{L_{\xi,\tau}^2},
\end{eqnarray*}
where in the last inequality we used the fact
$\norm{|t|^n\psi(t)}_{B^{1/2}_{2,1}}\leq\norm{|t|^n\psi(t)}_{H^{1}}\leq
C 2^n$. Finally, we consider the contribution of $I$.
\begin{eqnarray*}
I=\psi(t)\int_{|\tau|\leq 1}\sum_{n\geq
1}\frac{(it\tau)^n}{n!(i\tau+\epsilon|\xi|^{2\alpha})}\widehat{w}(\tau)d\tau.
\end{eqnarray*}
Thus, we get
\begin{eqnarray*}
&&\sum_{j=0}2^{j/2}\norm{\eta_k(\xi)P_j(I)(t)}_{L_{\xi,t}^2}\\
&\les&
\sum_{n\geq 1}\normo{\frac{t^n\psi(t)}{n!}}_{B^{1/2}_{2,1}}
\normo{\int_{|\tau|\leq1}\frac{|\tau|}{|i\tau+\epsilon|\xi|^{2\alpha}|}|\eta_k(\xi)\widehat{w}(\xi,\tau)|d\tau}_{L_\xi^2}\\
&\les&\sum_{j=0}2^{-j/2}\norm{\eta_k(\xi)\eta_j(\tau)\widehat{w}(\xi,
\tau)}_{L_{\xi,\tau}^2}.
\end{eqnarray*}
Therefore, we complete the proof of the proposition.
\end{proof}

In order to apply the standard fixed-point machinery, we next turn
to a bilinear estimate in $F^s$. The proof is divided into several
cases. We will use the estimate for the characterization multiplier
in Proposition \ref{pchar}. The first case is $low\times
high\rightarrow high$ interaction.
\begin{proposition}\label{punif}
If $k\geq 10$, $|k-k_2|\leq 5$, then for any $u\in F^s,\ v \in F^s$
\begin{equation}
\norm{(i+\tau-\xi^3)^{-1}\eta_k(\xi)i\xi\widehat{P_0u}*\widehat{P_{k_2}v}}_{X_k}\les
\norm{\widehat{P_0u}}_{X_0}\norm{\widehat{P_{k_2}v}}_{X_{k_2}}.
\end{equation}
\end{proposition}
\begin{proof}
For simplicity of notation we only prove the case that $k=k_2$,
since the other cases can be handled in the same way. From
definition of $X_k$, we get
\begin{eqnarray}
\norm{(i+\tau-\xi^3)^{-1}\eta_k(\xi)i\xi\widehat{P_0u}*\widehat{P_kv}}_{X_k}\les2^k\sum_{j,j_1,j_2\geq
0}2^{-j/2}\norm{1_{D_{k,j}}u_{0,j_1}*v_{k,j_2}}_2,
\end{eqnarray}
where
\[u_{0,j_1}=\eta_0(\xi)\eta_{j_1}(\tau-\xi^3)\widehat{u},\
v_{k,j_2}=\eta_k(\xi)\eta_{j_2}(\tau-\xi^3)\widehat{v}.\] Thus, in
view of definition it suffices to show that
\begin{equation}\label{eq:biunicase1}
\norm{1_{D_{k,j}}u_{0,j_1}*v_{k,j_2}}_2\les 2^{-k}
2^{(j_1+j_2)/2}\norm{u_{0,j_1}}_2\norm{v_{k,j_2}}_2.
\end{equation}
By duality and
$\xi_1^3+\xi_2^3-(\xi_1+\xi_2)^3=-3\xi_1\xi_2(\xi_1+\xi_2)$,
\eqref{eq:biunicase1} is equivalent to
\begin{eqnarray}
&&\aabs{\int \int u(\xi_1,\tau_1)v(\xi_2,\tau_2)g(\xi_1+\xi_2,
\tau_1+\tau_2-3\xi_1\xi_2(\xi_1+\xi_2))d\xi_1d\xi_2d\tau_1d\tau_2}\nonumber\\
&&\les \ 2^{-k} 2^{(j_1+j_2)/2}\norm{u}_2\norm{v}_2 \norm{g}_2
\end{eqnarray}
for any $u,v,g\in L^2$ supported in $I_0\times I_{j_1}$, $I_k\times
I_{j_2}$, $I_k\times I_{j}$ respectively. Therefore, it suffices to
show that
\begin{eqnarray}\label{eq:biunicase12}
&&\int_{|\xi_1|\leq 2} \int_{|\xi_2|\sim 2^k}
u(\xi_1)v(\xi_2)g(\xi_1+\xi_2,
-3\xi_1\xi_2(\xi_1+\xi_2))d\xi_1d\xi_2\nonumber\\
& & \les \ 2^{-k}\norm{u}_2\norm{v}_2 \norm{g}_2
\end{eqnarray}
for any $u,v,g\in L^2$ supported in $I_0$, $I_k$, $I_k\times
\widetilde{I}_{j_{max}}$ respectively where
$j_{max}=\max(j,j_1,j_2)$ and $\widetilde{I}_{j_{max}}=\cup_{l=-3}^3
I_{j_{max}+l}$.

Indeed, by changing the coordinates $\mu_1=\xi_1,\
\mu_2=\xi_1+\xi_2$, the left-side of \eqref{eq:biunicase12} is
bounded by
\begin{equation}\label{eq:biunicase13}
\int_{|\mu_1|\leq 2} \int_{|\mu_2|\sim 2^k}
u(\mu_1)v(\mu_2-\mu_1)g(\mu_2,
-3\mu_1(\mu_2-\mu_1)\mu_2)d\mu_1d\mu_2.
\end{equation}
Since in the integration area
\begin{equation}
\left|\frac{\partial}{\partial_{\mu_1}}[-3\mu_1(\mu_2-\mu_1)\mu_2]\right|\sim
2^{2k},
\end{equation}
then by Cauchy-Schwarz inequality we get
\begin{eqnarray}
\eqref{eq:biunicase13}&\les& \norm{u}_2\norm{v}_2 \norm{g(\mu_2,
-3\mu_1(\mu_2-\mu_1)\mu_2)}_{L_{|\mu_1|\leq 2, |\mu_2|\sim 2^k}^2}\nonumber\\
&\les&2^{-k} \norm{u}_2 \norm{v}_2 \norm{g}_2,
\end{eqnarray}
which completes the proof.
\end{proof}

\begin{proposition}
If $k\geq 10$, $|k-k_2|\leq 5$ and $1\leq k_1\leq k-9$. Then for any
$u, v\in F^s$
\begin{equation}
\norm{(i+\tau-\xi^3)^{-1}\eta_k(\xi)i\xi\widehat{P_{k_1}u}*\widehat{P_{k_2}v}}_{X_k}\les\
k^32^{-k/2}2^{-k_1}
\norm{\widehat{P_{k_1}u}}_{X_{k_1}}\norm{\widehat{P_{k_2}v}}_{X_{k_2}}.
\end{equation}
\end{proposition}
\begin{proof}
We only prove the case $k=k_2$. From the definition, we get
\begin{eqnarray}
\norm{(i+\tau-\xi^3)^{-1}\eta_k(\xi)i\xi\widehat{P_{k_1}u}*\widehat{P_kv}}_{X_k}\les2^k\sum_{j,j_1,j_2\geq
0}2^{-j/2}\norm{1_{D_{k,j}}u_{k_1,j_1}*v_{k,j_2}}_2,
\end{eqnarray}
where
\[u_{k_1,j_1}=\eta_{k_1}(\xi)\eta_{j_1}(\tau-\xi^3)\widehat{u},\
v_{k,j_2}=\eta_k(\xi)\eta_{j_2}(\tau-\xi^3)\widehat{v}.\] By
checking the support properties of the functions $u_{k_1,j_1},\
v_{k_2,j_2}$ and using the fact that
$|\xi_1^3+\xi_2^3-(\xi_1+\xi_2)^3|\sim 2^{2k+k_1}$, we get that
$1_{D_{k,j}}u_{k_1,j_1}*v_{k,j_2}\equiv 0$ unless $j_{max}\geq
2k+k_1-10$. Using \eqref{eq:charii}, we get
\begin{eqnarray}
&&2^k\sum_{j,j_1,j_2\geq 0}2^{-j/2}\norm{1_{D_{k,j}}u_{k_1,j_1}*v_{k,j_2}}_2\nonumber\\
&&\les \ 2^k\sum_{j,j_1,j_2\geq 0}2^{-j/2}2^{j_{min}/2}2^{-k/2}2^{-k_1/2}2^{j_{med}/2}\norm{u_{k_1,j_1}}_2\norm{v_{k,j_2}}_2\nonumber\\
&&\les \ 2^k\sum_{j_{max}\geq
2k+k_1-10}k^32^{-k/2}2^{-k_1/2}2^{-j_{max}/2}\norm{\widehat{P_{k_1}u}}_{X_{k_1}}\norm{\widehat{P_kv}}_{X_k}\nonumber\\
&&\les \
k^32^{-k/2}2^{-k_1}\norm{\widehat{P_{k_1}u}}_{X_{k_1}}\norm{\widehat{P_kv}}_{X_k},
\end{eqnarray}
which completes the proof of the proposition.
\end{proof}

The second case is $high\times high \rightarrow low$. This case is
the worst and where the condition is imposed. This is easy to be
seen, since $s\leq 0$ and  $\norm{u}_{F^s}, \norm{v}_{F^s}$ are
small for $u,v$ with very high frequency.
\begin{proposition}
If $k\geq 10$, $|k-k_2|\leq 5$, then for any $u,\ v \in F^s$
\begin{equation}\label{eq:biunicase2}
\norm{(i+\tau-\xi^3)^{-1}\eta_0(\xi)i\xi\widehat{P_{k}u}*\widehat{P_{k_2}v}}_{X_0}\les
\  k^32^{-3k/2}
\norm{\widehat{P_{k}u}}_{X_{k}}\norm{\widehat{P_{k_2}v}}_{X_{k_2}}.
\end{equation}
\end{proposition}
\begin{proof}
As before we assume $k=k_2$. From the definition, we get
\begin{eqnarray}
\norm{(i+\tau-\xi^3)^{-1}\eta_0(\xi)i\xi\widehat{P_{k}u}*\widehat{P_kv}}_{X_0}\les\sum_{k'=-\infty}^0
2^{k'}\sum_{j,j_1,j_2=0}2^{-j/2}\norm{1_{D_{k',j}}u_{k,j_1}*v_{k,j_2}}_2,
\end{eqnarray}
where
\begin{equation}\label{eq:ukj}
u_{k,j_1}=\eta_{k}(\xi)\eta_{j_1}(\tau-\xi^3)\widehat{u},\
v_{k,j_2}=\eta_k(\xi)\eta_{j_2}(\tau-\xi^3)\widehat{v}.
\end{equation}
 We may
assume that $k'\geq -10k$ and $j,j_1,j_2\leq 10k$. Otherwise, from
the following simple estimate which follows from H\"older's
inequality and Young's inequality
\[\norm{1_{D_{k',j}}u_{k,j_1}*v_{k,j_2}}_2\les 2^{j_{min}/2}2^{k'/2}\norm{u_{k,j_1}}_2\norm{v_{k,j_2}}_2\]
we immediately obtain \eqref{eq:biunicase2}. For the same reason as
in the proof of last proposition, we see that $j_{max}\geq
2k+k'-10$. Using \eqref{eq:charii}, we get
\begin{eqnarray}
&&\norm{(i+\tau-\xi^3)^{-1}\eta_0(\xi)i\xi\widehat{P_{k}u}*\widehat{P_kv}}_{X_0}\nonumber\\
&&\les \ \sum_{k'=-10k}^0
2^{k'}\sum_{j,j_1,j_2\geq 0}2^{-j/2}\norm{1_{D_{k',j}}u_{k,j_1}*v_{k,j_2}}_2\nonumber\\
&&\les \ \sum_{k'=-10k}^0\sum_{j,j_1,j_2\geq 0}2^{-j/2}
2^{k'}2^{j_{min}/2}2^{-k/2}2^{-k'/2}2^{j_{med}/2}\norm{u_{k,j_1}}_2\norm{v_{k,j_2}}_2\nonumber\\
&&\les \ \sum_{k'=-10k}^0\sum_{j_{max}\geq
2k+k'}k^22^{-k/2}2^{k'/2}2^{-j_{max}/2}\norm{\widehat{P_{k}u}}_{X_{k}}\norm{\widehat{P_kv}}_{X_k}\nonumber\\
&&\les \
k^32^{-3k/2}\norm{\widehat{P_{k}u}}_{X_{k}}\norm{\widehat{P_kv}}_{X_k}.
\end{eqnarray}
Therefore, we complete the proof of the proposition.
\end{proof}

\begin{proposition}
If $k\geq 10$,$|k-k_2|\leq 5$ and $1\leq k_1\leq k-9$, then for any
$u,\ v \in F^s$
\begin{equation}
\norm{(i+\tau-\xi^3)^{-1}\eta_{k_1}(\xi)i\xi\widehat{P_{k}u}*\widehat{P_{k_2}v}}_{X_{k_1}}\les
\  k^32^{-3k/2}
\norm{\widehat{P_{k}u}}_{X_{k}}\norm{\widehat{P_{k_2}v}}_{X_{k_2}}.
\end{equation}
\end{proposition}
\begin{proof}
As before we assume $k=k_2$. From the definition of $X_{k_1}$, we
get
\begin{eqnarray}\label{eq:highhigh2}
\norm{(i+\tau-\xi^3)^{-1}\eta_{k_1}(\xi)i\xi\widehat{P_{k}u}*\widehat{P_kv}}_{X_{k_1}}\les
\ 2^{k_1}\sum_{j,j_1,j_2\geq
0}2^{-j/2}\norm{1_{D_{k_1,j}}u_{k,j_1}*v_{k,j_2}}_2,
\end{eqnarray}
where $u_{k,j_1},v_{k,j_2}$ are as in \eqref{eq:ukj}. For the same
reason as before we have $j_{max}\geq 2k+k_1-10$ and we may assume
$j,j_1,j_2\leq 10k$. It follows from \eqref{eq:charii} that the
right-hand side of \eqref{eq:highhigh2} is bounded by
\begin{eqnarray*}
&&\sum_{j,j_1,j_2\geq 0}2^{-j/2}
2^{k_1}2^{j_{min}/2}2^{-k/2}2^{-{k_1}/2}2^{j_{med}/2}\norm{u_{k,j_1}}_2\norm{v_{k,j_2}}_2\nonumber\\
&&\les \ \sum_{j_{max}\geq
2k+k_1}k^22^{-k/2}2^{{k_1}/2}2^{-j_{max}/2}\norm{\widehat{P_{k}u}}_{X_{k}}\norm{\widehat{P_kv}}_{X_k}\les
\
k^32^{-3k/2}\norm{\widehat{P_{k}u}}_{X_{k}}\norm{\widehat{P_kv}}_{X_k}.
\end{eqnarray*}
Therefore we complete the proof of the proposition.
\end{proof}

\begin{proposition}
If $k\geq 10$, $|k-k_2|\leq 5$ and $k-9\leq k_1\leq k+10$, then for
any $u,\ v \in F^s$
\begin{equation}
\norm{(i+\tau-\xi^3)^{-1}\eta_{k_1}(\xi)i\xi\widehat{P_{k}u}*\widehat{P_{k_2}v}}_{X_{k_1}}\les
\  k^32^{-3k/4}
\norm{\widehat{P_{k}u}}_{X_{k}}\norm{\widehat{P_{k_2}v}}_{X_{k_2}}.
\end{equation}
\end{proposition}
\begin{proof}
As before we assume $k=k_2$. From the definition of $X_{k_1}$, we
get
\begin{eqnarray}\label{eq:highhigh3}
\norm{(i+\tau-\xi^3)^{-1}\eta_{k_1}(\xi)i\xi\widehat{P_{k}u}*\widehat{P_kv}}_{X_{k_1}}\les
\ 2^{k_1}\sum_{j,j_1,j_2\geq
0}2^{-j/2}\norm{1_{D_{k_1,j}}u_{k,j_1}*v_{k,j_2}}_2,
\end{eqnarray}
where $u_{k,j_1},v_{k,j_2}$ are as in \eqref{eq:ukj}. For the same
reason as before we have $j_{max}\geq 2k+k_1-10$ and we may assume
$j,j_1,j_2\leq 10k$. It follows from \eqref{eq:chari} that the
right-hand side of \eqref{eq:highhigh3} is bounded by
\begin{eqnarray*}
&&\sum_{j,j_1,j_2\geq 0}2^{-j/2}
2^{k_1}2^{j_{min}/2}2^{-k/4}2^{j_{med}/4}\norm{u_{k,j_1}}_2\norm{v_{k,j_2}}_2\les
\
k^32^{-3k/4}\norm{\widehat{P_{k}u}}_{X_{k}}\norm{\widehat{P_kv}}_{X_k},
\end{eqnarray*}
which completes the proof of the proposition.
\end{proof}

The final case is $low\times low \ra low$ interaction. Generally
speaking, this case is always easy to handle in many situations.
\begin{proposition}\label{punil}
If $0\leq k_1,k_2,k_3\leq 100$, then for any $u,\ v \in F^s$
\begin{equation}
\norm{(i+\tau-\xi^3)^{-1}\eta_{k_1}(\xi)i\xi\widehat{P_{k_2}u}*\widehat{P_{k_3}v}}_{X_{k_1}}\les
\norm{\widehat{P_{k_2}u}}_{X_{k_2}}\norm{\widehat{P_{k_3}v}}_{X_{k_3}}.
\end{equation}
\end{proposition}
\begin{proof}
From the definition of $X_{k_1}$, we get that
\begin{eqnarray}\label{eq:highhigh3}
\norm{(i+\tau-\xi^3)^{-1}\eta_{k_1}(\xi)i\xi\widehat{P_{k_2}u}*\widehat{P_{k_3}v}}_{X_{k_1}}\les
\ 2^{k_1}\sum_{j,j_1,j_2\geq
0}2^{-j/2}\norm{1_{D_{k_1,j}}u_{k_2,j_1}*v_{k_3,j_2}}_2,
\end{eqnarray}
where $u_{k_2,j_1},v_{k_3,j_2}$ are as in \eqref{eq:ukj}. By
checking the support properties of the function
$u_{k_2,j_1},v_{k_3,j_2}$, we get that
$1_{D_{k_1,j}}u_{k_2,j_1}*v_{k_3,j_2}\equiv 0$ unless
$|j_{max}-j_{med}|\leq 10$ or $j_{max}\leq 1000$ where $j_{max},
j_{med}$ are the maximum and median of $j,j_1,j_2$ respectively. It
follows immediately from Young's inequality that
\begin{equation}
\norm{1_{D_{k,j}} u_{k_1,j_1}*v_{k_2,j_2}}_{L_{\xi,\tau}^2}\les
2^{k_i}2^{j_i}\norm{u_{k_1,j_1}}_2\norm{v_{k_2,j_2}}_2,\ i=1,2.
\end{equation}
From definition and summing in $j_i$, we complete the proof of the
proposition.
\end{proof}

With these propositions in hand, we are able to prove the bilinear
estimate. The idea is to decompose the bilinear product using
para-product, and then divide it into many cases according to the
interactions. Finally we use discrete Young's inequality.
\begin{proposition}\label{p412}
Fix any $s\in (-3/4, 0]$, $\forall s\leq \sigma \leq 0$, there
exists $C>0$ such that for any $u,v\in F^\sigma$,
\begin{equation}\label{eq:biuni}
\norm{\partial_x(uv)}_{N^\sigma}\leq
C(\norm{u}_{F^s}\norm{v}_{F^\sigma}+\norm{v}_{F^s}\norm{u}_{F^\sigma}).
\end{equation}
\end{proposition}
\begin{proof}
In view of definition, we get that
\begin{equation}
\norm{\partial_x(uv)}_{N^\sigma}^2=\sum_{k_3\in \Z_+}2^{2\sigma
k_3}\norm{(i+\tau-\xi^3)^{-1}\eta_{k_3}(\xi)i\xi\widehat{u}*\widehat{v}}_{X_{k_3}}^2.
\end{equation}
We decompose $\widehat{u}, \widehat{v}$ and get
\begin{eqnarray}\label{eq:biunipara}
\norm{(i+\tau-\xi^3)^{-1}\eta_{k_3}(\xi)i\xi\widehat{u}*\widehat{v}}_{X_{k_3}}\les
\sum_{k_1,k_2\in \Z_+}
\norm{(i+\tau-\xi^3)^{-1}\eta_{k_3}(\xi)i\xi\widehat{P_{k_1}u}*\widehat{P_{k_2}v}}_{X_{k_3}}.
\end{eqnarray}
By checking the support properties we get that
$\eta_{k_3}(\xi)\widehat{P_{k_1}u}*\widehat{P_{k_2}v}\equiv 0$
unless $|k_{max}-k_{med}|\leq 5$ where $k_{max}, k_{med}$ are the
maximum and median of $k_1,k_2,k_3$ respectively. We may assume that
$k_1\leq k_2$ from symmetry. By dividing the summation into
$high\times high$, $high\times low$ four parts, we get that the
right-hand side of \eqref{eq:biunipara} is bounded by
\begin{eqnarray}
\big(\sum_{j=1}^4\sum_{k_1,k_2\in A_j}\big)
\norm{(i+\tau-\xi^3)^{-1}\eta_{k_3}(\xi)i\xi\widehat{P_{k_1}u}*\widehat{P_{k_2}v}}_{X_{k_3}},
\end{eqnarray}
where $A_j$, $j=1,2,3,4$ are defined by
\begin{eqnarray*}
&&A_1=\{k_2\geq 10, |k_2-k_3|\leq 5, k_1\leq k_2-10\};\\
&&A_2=\{k_2\geq 10, |k_2-k_3|\leq 5, k_2-9\leq k_1\leq k_2+10\};\\
&&A_3=\{k_2\geq 10, |k_2-k_1|\leq 5, k_3\leq k_1-10\};\\
&&A_4=\{k_1,k_2,k_3\leq 100\}.
\end{eqnarray*}
Therefore, \eqref{eq:biuni} from the Proposition
\ref{punif}-\ref{punil}, discrete Young's inequality and the
assumption that $s>-3/4$.
\end{proof}

We next show \eqref{eq:kdvb} is uniformly (on $0<\epsilon \leq 1$)
locally well-posed in $H^s$, $-3/4<s\leq 0$. The procedure is quite
standard. See \cite{KPV}, for instance. By the scaling
\eqref{eq:scaling}, we see that $u$ solves \eqref{eq:kdvb} if and
only if $u_\lambda (x,t)= \lambda^2 u(\lambda x, \lambda^3 t)$
solves
\begin{eqnarray}\label{eq:kdvb-m}
\partial_t u_\lambda + \partial^3_x u_{\lambda}+\epsilon \lambda^{3-2\alpha} |\partial_x|^{2\alpha}u_\lambda+\partial_x (u^2_\lambda)=0, \ \
u_\lambda (0)= \lambda^2\phi (\lambda \, \cdot).
\end{eqnarray}
Since $-3/4<s\leq 0$,
\begin{equation}
\norm{\lambda^2\phi(\lambda
x)}_{H^s}=O(\lambda^{3/2+s}\norm{\phi}_{H^s}) \quad \mbox{as }
\lambda\rightarrow 0,
\end{equation}
thus we can first restrict ourselves to considering \eqref{eq:kdvb}
with data $\phi$ satisfying
\begin{equation}
\norm{\phi}_{H^s}=r\ll 1.
\end{equation}

As in the last section, we will mainly work on the integral equation
\eqref{eq:trunintekdvbuni}. We define the operator
\begin{equation}
\Phi_{\phi}(u)=\psi(t)W_\epsilon^\alpha(t)\phi- \psi(t)L
\big(\partial_x (\psi^2 u^2)\big),
\end{equation}
where $L$ is defined by \eqref{defL}. We will prove that $\Phi_\phi
(\cdot)$ is a contraction  mapping from
\begin{equation}
\mathcal{B}=\{w\in F^s:\ \norm{w}_{F^s}\leq 2cr\}
\end{equation}
into itself. From Propositions \ref{p42}, \ref{p43} and \ref{p45} we
get if $w\in \mathcal{B}$, then
\begin{eqnarray}
\norm{\Phi_\phi(w)}_{F^s}&\leq&
c\norm{\phi}_{H^s}+\norm{\partial_x(\psi(t)^2w^2(\cdot, t))}_{N^s}\nonumber\\
&\leq& cr+c\norm{w}_{F^s}^2\leq cr+c(2cr)^2\leq 2cr,
\end{eqnarray}
provided $r$ satisfies $4c^2r\leq 1/2$. Similarly, for $w, h\in
\mathcal{B}$
\begin{eqnarray}
\norm{\Phi_\phi(w)-\Phi_\phi(h)}_{F^s}
&\leq& c\normo{ L \partial_x(\psi^2(\tau)(u^2(\tau)-h^2(\tau)))}_{F^s}\nonumber\\
&\leq&c\norm{w+h}_{F^s}\norm{w-h}_{F^s}\nonumber\\
&\leq&4c^2r\norm{w-h}_{F^s}\leq \frac{1}{2}\norm{w-h}_{F^s}.
\end{eqnarray}
Thus $\Phi_\phi(\cdot)$ is a contraction. There exists a unique
$u\in \mathcal{B}$ such that
\begin{equation}
u=\psi(t)W_\epsilon^\alpha(t)\phi- \psi(t)L \big(\partial_x (\psi^2
u^2)\big).
\end{equation}
Hence $u$ solves the integral equation \eqref{eq:intekdvb} in the
time interval $[0,1]$.

We prove now that $u\in X^{1/2,s,\alpha}$. Indeed, from the slightly
modified argument as the proof for Proposition 2.1, 2.3 \cite{MR},
we can show that
\begin{eqnarray*}
&&\norm{\psi(t)W_\epsilon^\alpha(t)\phi}_{X^{1/2,s,\alpha}}\les
\norm{\phi}_{H^s};\\
&&\norm{\psi(t)L(v)}_{X^{1/2,s,\alpha}}\les
\norm{v}_{X^{-1/2,s,\alpha}}+\brk{\int\jb{\xi}^{2s}\big(\int
\frac{|\widehat{v}(\tau)|}{\jb{i\tau+\epsilon|\xi|^{2\alpha}}}d\tau\big)^2d\xi}^{1/2}\les
\norm{v}_{N^s},
\end{eqnarray*}
which then imply $u\in X^{1/2,s,\alpha}$, as desired. For general
$\phi \in H^s$, by using the scaling \eqref{eq:scaling} and the
uniqueness in Theorem \ref{t11}, we immediately obtain that Theorem
\ref{t12} holds for a small $T=T(\norm{\phi}_{H^s})>0$.

\section{Uniform global well-posedness for KdV-B equation}
In this section we will extend the uniform local solution obtained
in the last section to a uniform global solution. The standard way
is to use conservation law. Let $u$ be a smooth solution of
\eqref{eq:kdvb}, multiply $u$ and integrate, then we get
\begin{equation}\label{eq:L2law}
\half{1}\norm{u(t)}_2^2+\epsilon \int_0^t\norm{\Lambda^\alpha
u(\tau)}_2^2d\tau=\half{1}\norm{\phi}_2^2.
\end{equation}
By a standard limit argument, \eqref{eq:L2law} holds for
$L^2$-strong solution. Thus if $\phi\in L^2$, then we get that
\eqref{eq:kdvb} is uniformly globally well-posed.

For $\phi\in H^s$ with $-3/4<s<0$, there is no such conservation
law. We will follow the idea in \cite{Tao2} (I-method) to extend the
solution. Let $m: \R^k \rightarrow \C$ be a function. We say $m$ is
symmetric if $m(\xi_1,\ldots, \xi_k)=m(\sigma(\xi_1,\ldots, \xi_k))$
for all $\sigma \in S_k$, the group of all permutations on $k$
objects. The symmetrization of $m$ is the function
\begin{equation}
[m]_{sym}(\xi_1,\xi_2,\ldots, \xi_k)=\rev{k!}\sum_{\sigma\in
S_k}m(\sigma(\xi_1,\xi_2,\ldots,\xi_k)).
\end{equation}
We define a $k-linear$ functional associated to the multiplier $m$
acting on $k$ functions $u_1,\ldots,u_k$,
\begin{equation}
\Lambda_k(m;u_1,\ldots,u_k)=\int_{\xi_1+\ldots+\xi_k=0}m(\xi_1,\ldots,\xi_k)\widehat{u_1}(\xi_1)\ldots
\widehat{u_k}(\xi_k).
\end{equation}
We will often apply $\Lambda_k$ to $k$ copies of the same function
$u$. $\Lambda_k(m;u,\ldots,u)$ may simply be written $\Lambda_k(m)$.
By the symmetry of the measure on hyperplane, we have
$\Lambda_k(m)=\Lambda_k([m]_{sym})$.

The following statement may be directly verified by using the KdV-B
equation \eqref{eq:kdvb}. Compared to the KdV equation, the KdV-B
equation has one more term caused by the dissipation.
\begin{proposition}
Suppose $u$ satisfies the KdV-B equation \eqref{eq:kdvb} and that
$m$ is a symmetric function. Then
\begin{equation}\label{eq:menergy}
\frac{d}{dt}\Lambda_k(m)=\Lambda_k(mh_k)-\epsilon
\Lambda_k(m\beta_{\alpha,k})-i\half k
\Lambda_{k+1}(m(\xi_1,\ldots,\xi_{k-1},\xi_k+\xi_{k+1})(\xi_k+\xi_{k+1})),
\end{equation}
where
\[h_k=i(\xi_1^3+\xi_2^3+\ldots+\xi_k^3), \quad \beta_{\alpha,k}=|\xi_1|^{2\alpha}+|\xi_2|^{2\alpha}+\ldots+|\xi_k|^{2\alpha}.\]
\end{proposition}

We follow the I-method \cite{Tao2} to define a set of modified
energies. Let $m:\R\rightarrow \R$ be an arbitrary even $\R$-valued
function and define the operator by
\begin{equation}
\widehat{If}(\xi)=m(\xi)\widehat{f}(\xi).
\end{equation}
We define the modified energy $E_I^2(t)$ by
\begin{equation}
E_I^2(t)=\norm{Iu(t)}_{L^2}^2.
\end{equation}
By Plancherel and the fact that $m$ and $u$ are $\R$-valued, and $m$
is even,
\[E_I^2(t)=\Lambda_2(m(\xi_1)m(\xi_2)).\]
Using \eqref{eq:menergy}, we have
\begin{eqnarray}
\frac{d}{dt}E_I^2(t)&=&\Lambda_2(m(\xi_1)m(\xi_2)h_2)-\epsilon\Lambda_2(m(\xi_1)m(\xi_2)\beta_{\alpha,2})\nonumber\\
&&-i\Lambda_3(m(\xi_1)m(\xi_2+\xi_3)(\xi_2+\xi_3)).
\end{eqnarray}
The first term vanishes. The second term is non-positive, hence
good. We symmetrize the third term to get
\begin{equation}
\frac{d}{dt}E_I^2(t)=-\epsilon\Lambda_2(m(\xi_1)m(\xi_2)\beta_{\alpha,2})+\Lambda_3(-i[m(\xi_1)m(\xi_2+\xi_3)(\xi_2+\xi_3)]_{sym}).
\end{equation}
Let us denote
\begin{equation}
M_3(\xi_1,\xi_2,\xi_3)=-i[m(\xi_1)m(\xi_2+\xi_3)(\xi_2+\xi_3)]_{sym}.
\end{equation}
Form the new modified energy
\[E_I^3(t)=E_I^2(t)+\Lambda_3(\sigma_3)\]
where the symmetric function $\sigma_3$ will be chosen momentarily
to achieve a cancellation. Applying \eqref{eq:menergy} gives
\begin{eqnarray}\label{eq:E3}
\frac{d}{dt}E_I^3(t)&=&-\epsilon\Lambda_2(m(\xi_1)m(\xi_2)\beta_{\alpha,2})+\Lambda_3(M_3)\nonumber\\
&&+\Lambda_3(\sigma_3h_3)-\epsilon
\Lambda_3(\sigma_3\beta_{\alpha,3})-\half 3
i\Lambda_4(\sigma_3(\xi_1,\xi_2,\xi_3+\xi_4)(\xi_3+\xi_4)).
\end{eqnarray}
Compared to the KdV case \cite{Tao2}, there is one more term to
cancel, so we choose
\begin{equation}
\sigma_3=-\frac{M_3}{h_3-\epsilon \beta_{\alpha, 3}}
\end{equation}
to force the three $\Lambda_3$ terms in \eqref{eq:E3} to cancel.
Hence if we denote
\begin{equation}
M_4(\xi_1,\xi_2,\xi_3,\xi_4)=-i\half
3[\sigma_3(\xi_1,\xi_2,\xi_3+\xi_4)(\xi_3+\xi_4)]_{sym}
\end{equation}
then
\begin{eqnarray}
\frac{d}{dt}E_I^3(t)=-\epsilon\Lambda_2(m(\xi_1)m(\xi_2)\beta_{\alpha,2})+\Lambda_4(M_4).
\end{eqnarray}
Similarly defining
\[E_I^4(t)=E_I^3(t)+\Lambda_4(\sigma_4)\]
with
\begin{equation}
\sigma_4=-\frac{M_4}{h_4-\epsilon \beta_{\alpha,4}},
\end{equation}
we obtain
\begin{eqnarray}
\frac{d}{dt}E_I^4(t)=-\epsilon\Lambda_2(m(\xi_1)m(\xi_2)\beta_{\alpha,2})+\Lambda_5(M_5)
\end{eqnarray}
where
\begin{equation}
M_5(\xi_1,\ldots,\xi_5)=-2i[\sigma_4(\xi_1,\xi_2,\xi_3,\xi_4+\xi_5)(\xi_4+\xi_5)]_{sym}.
\end{equation}

Now we give pointwise bounds for the multipliers. We will only be
interested in  the value of the multiplier on the hyperplane
$\xi_1+\xi_2+\ldots+\xi_k=0$. There is a flexibility of choosing the
multiplier $m$. In application, we consider $m(\xi)$ is smooth,
monotone, and of the form
\begin{eqnarray}\label{eq:Imul}
m(\xi)=\left \{
\begin{array}{l}
1, \quad |\xi|<N,\\
N^{-s}|\xi|^s,\quad |\xi|>2N.
\end{array}
\right.
\end{eqnarray}
It is easy to see that if $m$ is of the form \eqref{eq:Imul}, then
$m^2$ satisfies
\begin{eqnarray}\label{eq:eImul}
&&m^2(\xi)\sim m^2(\xi') \mbox{ for } |\xi|\sim|\xi'|,\nonumber\\
&&(m^2)'(\xi)=O(\frac{m^2(\xi)}{|\xi|}),\nonumber\\
&&(m^2)''(\xi)=O(\frac{m^2(\xi)}{|\xi|^2}).
\end{eqnarray}

We will need two mean value formulas which follow immediately from
the fundamental theorem of calculus. If $|\eta|,|\lambda|\ll |\xi|$,
then we have
\begin{equation}\label{eq:mvt}
|a(\xi+\eta)-a(\xi)|\les |\eta|\sup_{|\xi'|\sim |\xi|}|a'(\xi')|,
\end{equation}
and the double mean value formula that
\begin{equation}\label{eq:dmvt}
|a(\xi+\eta+\lambda)-a(\xi+\eta)-a(\xi+\lambda)+a(\xi)|\les
|\eta||\lambda|\sup_{|\xi'|\sim |\xi|}|a''(\xi')|.
\end{equation}

\begin{proposition}
If $m$ is of the form \eqref{eq:Imul}, then for each dyadic
$\lambda\leq \mu$ there is an extension of $\sigma_3$ from the
diagonal set
\[\{(\xi_1,\xi_2,\xi_3)\in \Gamma_3(\R), |\xi_1|\sim \lambda, |\xi_2|, |\xi_3|\sim \mu\}\]
to the full dyadic set
\[\{(\xi_1,\xi_2,\xi_3)\in \R^3, |\xi_1|\sim \lambda, |\xi_2|, |\xi_3|\sim \mu\}\]
which satisfies
\begin{equation}\label{eq:m3}
|\partial_1^{\beta_1}\partial_2^{\beta_2}\partial_3^{\beta_3}\sigma_3(\xi_1,\xi_2,\xi_3)|\leq
C m^2(\lambda)\mu^{-2}\lambda^{-\beta_1}\mu^{-\beta_2-\beta_3},
\end{equation}
where $C$ is independent of $\epsilon$.
\end{proposition}
\begin{proof}
Since on the hyperplane $\xi_1+\xi_2+\xi_3=0$,
\[h_3=i(\xi_1^3+\xi_2^3+\xi_3^3)=3i\xi_1\xi_2\xi_3\]
is with a size about $\lambda \mu^2$ and
\[M_3(\xi_1,\xi_2,\xi_3)=-i[m(\xi_1)m(\xi_2+\xi_3)(\xi_2+\xi_3)]_{sym}=i(m^2(\xi_1)\xi_1+m^2(\xi_2)\xi_2+m^2(\xi_3)\xi_3),\]
if $\lambda \sim \mu$, we extend $\sigma_3$ by setting
\begin{equation}
\sigma_3(\xi_1,\xi_2,\xi_3)=-\frac{i(m^2(\xi_1)\xi_1+m^2(\xi_2)\xi_2+m^2(\xi_3)\xi_3)}{3i\xi_1\xi_2\xi_3-\epsilon(|\xi_1|^{2\alpha}+|\xi_2|^{2\alpha}+|\xi_3|^{2\alpha})},
\end{equation}
and if $\lambda\ll \mu$, we extend $\sigma_3$ by setting
\begin{equation}
\sigma_3(\xi_1,\xi_2,\xi_3)=-\frac{i(m^2(\xi_1)\xi_1+m^2(\xi_2)\xi_2-m^2(\xi_1+\xi_2)(\xi_1+\xi_2))}{3i\xi_1\xi_2\xi_3-\epsilon(|\xi_1|^{2\alpha}+|\xi_2|^{2\alpha}+|\xi_3|^{2\alpha})}.
\end{equation}
 From \eqref{eq:mvt} and \eqref{eq:eImul}, we see that \eqref{eq:m3} holds.
\end{proof}

We define on the hyperplane $\{(\xi_1,\xi_2,\xi_3)\in \Gamma_3(\R),
|\xi_1|\approx \lambda, |\xi_2|, |\xi_3|\approx \mu\}$
\begin{equation}
\sigma_3^{-}(\xi_1,\xi_2,\xi_3)=-\frac{i(m^2(\xi_1)\xi_1+m^2(\xi_2)\xi_2+m^2(\xi_3)\xi_3)}{3i\xi_1\xi_2\xi_3+\epsilon(|\xi_1|^{2\alpha}+|\xi_2|^{2\alpha}+|\xi_3|^{2\alpha})},
\end{equation}
and extend it as for $\sigma_3$. Then \eqref{eq:m3} also holds for
$\sigma_3^{-}$, and on the hyperplane $\xi_1+\xi_2+\xi_3=0$ we get
\begin{equation}\label{eq:diffm3}
|\sigma_3(\xi_1,\xi_2,\xi_3)-\sigma_3^{-}(\xi_1,\xi_2,\xi_3)|\les
\frac{\epsilon
|\xi|_{max}^{2\alpha}m^2(|\xi|_{min})|\xi|_{min}}{(\xi_1\xi_2\xi_3)^2+\epsilon^2|\xi|_{max}^{4\alpha}},
\end{equation}
where
\[|\xi|_{max}=\max(|\xi_1|,|\xi_2|,|\xi_3|), \quad |\xi|_{min}=\min(|\xi_1|,|\xi_2|,|\xi_3|).\]

Now we give the pointwise bounds for $\sigma_4$ which is key to
estimate the growth of $E^4_I(t)$. It has the same bound as in the
KdV case.
\begin{proposition}
Assume $m$ is of the form \eqref{eq:Imul}. In the region where
$|\xi_i|\sim N_i,|\xi_j+\xi_k|\sim N_{jk}$ for $N_i, N_{jk}$ dyadic,
\begin{equation}\label{eq:m4}
\frac{|M_4(\xi_1,\xi_2,\xi_3,\xi_4)|}{|h_4-\epsilon
\beta_{\alpha,4}|}\les
\frac{m^2(\min(N_i,N_{jk}))}{(N+N_1)(N+N_2)(N+N_3)(N+N_4)}.
\end{equation}
\end{proposition}
\begin{proof}
From symmetry, we can assume that $N_1\geq N_2\geq N_3\geq N_4$.
Since $\xi_1+\xi_2+\xi_3+\xi_4=0$, then $N_1\sim N_2$. We can also
assume that $N_1\sim N_2 \ges N$, otherwise $M_4$ vanishes, since
$m^2(\xi)=1$ if $|\xi|\leq N$. If $\max(N_{12},N_{13},N_{14})\ll
N_1$, then $\xi_3\approx-\xi_1,\ \xi_4\approx -\xi_1$, which
contradicts that $\xi_1+\xi_2+\xi_3+\xi_4=0$. Hence we get
$\max(N_{12},N_{13},N_{14})\sim N_1$. The right side of
\eqref{eq:m4} may be reexpressed as
\begin{equation}
\frac{m^2(\min(N_i,N_{jk}))}{{N_1}^2(N+N_3)(N+N_4)}.
\end{equation}

Since $\xi_1+\xi_2+\xi_3+\xi_4=0$, then
$h_4=\xi_1^3+\xi_2^3+\xi_3^3+\xi_4^3=3(\xi_1+\xi_2)(\xi_1+\xi_3)(\xi_1+\xi_4)$,
and we can write that
\begin{eqnarray}\label{eq:rm4}
CM_4(\xi_1,\xi_2,\xi_3,\xi_4)&=&[\sigma_3(\xi_1,\xi_2,\xi_3+\xi_4)(\xi_3+\xi_4)]_{sym}\nonumber\\
&=&\sigma_3(\xi_1,\xi_2,\xi_3+\xi_4)(\xi_3+\xi_4)+\sigma_3(\xi_1,\xi_3,\xi_2+\xi_4)(\xi_2+\xi_4)\nonumber\\
&&+\sigma_3(\xi_1,\xi_4,\xi_2+\xi_3)(\xi_2+\xi_3)+\sigma_3(\xi_2,\xi_3,\xi_1+\xi_4)(\xi_1+\xi_4)\nonumber\\
&&+\sigma_3(\xi_2,\xi_4,\xi_1+\xi_3)(\xi_1+\xi_3)+\sigma_3(\xi_3,\xi_4,\xi_1+\xi_2)(\xi_1+\xi_2)\nonumber\\
&=&[\sigma_3(\xi_1,\xi_2,\xi_3+\xi_4)-\sigma_3^{-}(-\xi_3,-\xi_4,\xi_3+\xi_4)](\xi_3+\xi_4)\nonumber\\
&&+[\sigma_3(\xi_1,\xi_3,\xi_2+\xi_4)-\sigma_3^{-}(-\xi_2,-\xi_4,\xi_2+\xi_4)](\xi_2+\xi_4)\nonumber\\
&&+[\sigma_3(\xi_1,\xi_4,\xi_2+\xi_3)-\sigma_3^{-}(-\xi_2,-\xi_3,\xi_2+\xi_3)](\xi_2+\xi_3)\nonumber\\
&=&I+II+III.
\end{eqnarray}
The bound \eqref{eq:m4} will follow from case by case analysis.

{\bf Case 1.} $|N_4|\ges \half{N}$.

{\bf Case 1a.} $N_{12}, N_{13}, N_{14}\ges N_1$.

For this case, we just use \eqref{eq:m3}, then we get
\begin{equation}
\frac{|M_4(\xi_1,\xi_2,\xi_3,\xi_4)|}{|h_4-\epsilon
\beta_{\alpha,4}|}\les\frac{|M_4(\xi_1,\xi_2,\xi_3,\xi_4)|}{|h_4|}\les
\frac{m^2(N_4)}{N_1N_2N_3N_4},
\end{equation}
which is acceptable.

{\bf Case 1b.} $N_{12}\ll N_1$, $N_{13}\ges N_1$, $N_{14}\ges N_1$.

Contribution of I. We just use \eqref{eq:m3}, then we get
\begin{equation}
\frac{|I|}{|h_4-\epsilon \beta_{\alpha,4}|}\les\frac{|I|}{|h_4|}\les
\frac{m^2(\min(N_4, N_{12}))}{N_1N_2N_3N_4},
\end{equation}
which is acceptable.

Contribution of II. We first write
\begin{eqnarray}
II&=&[\sigma_3(\xi_1,\xi_3,\xi_2+\xi_4)-\sigma_3^{-}(-\xi_2,-\xi_4,\xi_2+\xi_4)](\xi_2+\xi_4)\nonumber\\
&=&[\sigma_3(\xi_1,\xi_3,\xi_2+\xi_4)-\sigma_3^{-}(\xi_1,\xi_3,\xi_2+\xi_4)](\xi_2+\xi_4)\nonumber\\
&&+[\sigma_3^{-}(\xi_1,\xi_3,\xi_2+\xi_4)-\sigma_3^{-}(-\xi_2,-\xi_4,\xi_2+\xi_4)](\xi_2+\xi_4)\nonumber\\
&=&II_1+II_2.
\end{eqnarray}
Then from \eqref{eq:diffm3} we get
\begin{eqnarray}
\frac{II_1}{|h_4-\epsilon \beta_{\alpha,4}|}\les
\frac{II_1}{|\epsilon
\beta_{\alpha,4}|}\les\frac{m^2(N_4)}{N_1N_1N_1N_3}.
\end{eqnarray}
We now consider $II_2$. If $N_{12}\ges N_3$, then using
\eqref{eq:mvt} and \eqref{eq:m3}, or else if $N_{12}\ll N_3$, then
using \eqref{eq:mvt} twice and \eqref{eq:m3}, then
\begin{eqnarray}
\frac{II_2}{|h_4-\epsilon \beta_{\alpha,4}|}\les
\frac{II_2}{h_4}\les\frac{m^2(N_4)}{N_1N_1N_1N_3}.
\end{eqnarray}

Contribution of III. This is identical to II.

{\bf Case 1c.} $N_{12}\ll N_1$, $N_{13}\ll N_1$, $N_{14}\ges N_1$.

Since $N_{12}\ll N_1$, $N_{13}\ll N_1$, then $N_1\sim N_2\sim
N_3\sim N_4$.

Contribution of I. We first write
\begin{eqnarray}
I&=&[\sigma_3(\xi_1,\xi_2,\xi_3+\xi_4)-\sigma_3^{-}(\xi_1,\xi_2,\xi_3+\xi_4)](\xi_3+\xi_4)\nonumber\\
&&+[\sigma_3^{-}(\xi_1,\xi_2,\xi_3+\xi_4)-\sigma_3^{-}(-\xi_3,\xi_2,\xi_3+\xi_4)](\xi_3+\xi_4)\nonumber\\
&&+[\sigma_3^{-}(-\xi_3,\xi_2,\xi_3+\xi_4)-\sigma_3^{-}(-\xi_3,-\xi_4,\xi_3+\xi_4)](\xi_3+\xi_4)\nonumber\\
&=&I_1+I_2+I_3.
\end{eqnarray}
We use \eqref{eq:diffm3} for the first term and \eqref{eq:m3},
\eqref{eq:mvt} for the last two terms, then we get
\begin{eqnarray}
\frac{I}{|h_4-\epsilon \beta_{\alpha,4}|} \les \frac{I_1}{|\epsilon
\beta_{\alpha,4}|}+\frac{I_2}{|h_4|}+\frac{I_3}{|h_4|}\les\frac{m^2(N_{12})}{N_1^4}.
\end{eqnarray}

Contribution of II. This is identical to I.

Contribution of III. We first write
\begin{eqnarray}
III&=&[\sigma_3(\xi_1,\xi_4,\xi_2+\xi_3)-\sigma_3^{-}(-\xi_2,-\xi_3,\xi_2+\xi_3)](\xi_2+\xi_3)\nonumber\\
&=&[\sigma_3(\xi_1,\xi_4,\xi_2+\xi_3)-\sigma_3^{-}(\xi_1,\xi_4,\xi_2+\xi_3)](\xi_2+\xi_3)\nonumber\\
&&+1/2[\sigma_3^{-}(\xi_1,\xi_4,\xi_2+\xi_3)-\sigma_3^{-}(-\xi_2,-\xi_3,\xi_2+\xi_3)\nonumber\\
&&-\sigma_3^{-}(-\xi_3,-\xi_2,\xi_2+\xi_3)+\sigma_3^{-}(\xi_4,\xi_1,\xi_2+\xi_3)](\xi_2+\xi_3)\nonumber\\
&=&III_1+III_2.
\end{eqnarray}
We use \eqref{eq:diffm3} for the first term and \eqref{eq:dmvt} four
times for the second term, then we get
\begin{equation}
\frac{III}{|h_4-\epsilon \beta_{\alpha,4}|} \les
\frac{III_1}{|\epsilon
\beta_{\alpha,4}|}+\frac{III_2}{|h_4|}\les\frac{m^2(N_{1})}{N_1^4}.
\end{equation}

{\bf Case 1d.} $N_{12}\ll N_1$, $N_{13}\ges N_1$, $N_{14}\ll N_1$.

This case is identical to Case 1c.

{\bf Case 2.} $N_4\ll N/2$.

In this case we have $m^2(\min(N_i,N_{jk}))=1$, and $N_{13}\sim
|\xi_1+\xi_3|=|\xi_2+\xi_4|\sim N_1$. We discuss this case in the
following two subcases.

{\bf Case 2a.} $N_1/4>N_{12}\ges N/2$.

Since $N_4\ll N/2$ and $|\xi_3+\xi_4|=|\xi_1+\xi_2|\ges N/2$, then
$N_3\ges N/2$. From $|h_4|\sim N_{12}N_1^2$, then we bound the six
terms in \eqref{eq:rm4} respectively, and get
\begin{equation}
\frac{|M_4|}{|h_4-\epsilon\beta_{\alpha,4}|}\les
\frac{|M_4|}{|h_4|}\les \frac{1}{N_1^2N_3N},
\end{equation}
which is acceptable.

{\bf Case 2b.} $N_{12}\ll N/2$.

Since $N_{12}=N_{34}\ll N/2$ and $N_4\ll N/2$, then we must have
$N_3\ll N/2$, and $N_{13}\sim N_{14}\sim N_1$.

Contribution of I. Since $N_3, N_4, N_{34}\ll N/2$, then we have
$\sigma_3^{-}(-\xi_3,-\xi_4,\xi_3+\xi_4)=0$. Thus it follows from
\eqref{eq:m3} that
\begin{equation}
\frac{|I|}{|h_4-\epsilon\beta_{\alpha,4}|}\les
\frac{|\sigma_3(\xi_1,\xi_2,\xi_3+\xi_4)|}{N_1^2}\les
\frac{1}{N_1^4}.
\end{equation}

Contribution of II and III. We have two items of $N_3, N_4, N_{12}$
in the denominator, which will cause a problem. Thus we can't deal
with II and III separately, but we need to exploit the cancelation
between II and III. We rewrite
\begin{eqnarray}
II+III&=&[\sigma_3(\xi_1,\xi_3,\xi_2+\xi_4)-\sigma_3^{-}(-\xi_2,-\xi_4,\xi_2+\xi_4)](\xi_2+\xi_4)\nonumber\\
&&+[\sigma_3(\xi_1,\xi_4,\xi_2+\xi_3)-\sigma_3^{-}(-\xi_2,-\xi_3,\xi_2+\xi_3)](\xi_2+\xi_3)\nonumber\\
&=&[\sigma_3(\xi_1,\xi_3,\xi_2+\xi_4)-\sigma_3^{-}(-\xi_2,-\xi_4,\xi_2+\xi_4)]\xi_4\nonumber\\
&&+[\sigma_3(\xi_1,\xi_4,\xi_2+\xi_3)-\sigma_3^{-}(-\xi_2,-\xi_3,\xi_2+\xi_3)]\xi_3\nonumber\\
&&+[\sigma_3(\xi_1,\xi_3,\xi_2+\xi_4)-\sigma_3^{-}(-\xi_2,-\xi_4,\xi_2+\xi_4)\nonumber\\
&&\quad+\sigma_3(\xi_1,\xi_4,\xi_2+\xi_3)-\sigma_3^{-}(-\xi_2,-\xi_3,\xi_2+\xi_3)]\xi_2\nonumber\\
&=&J_1+J_2+J_3.
\end{eqnarray}
We first consider $J_1$. From
\begin{eqnarray}
\frac{|J_1|}{|h_4-\epsilon\beta_{\alpha,4}|}&\leq&
\frac{|[\sigma_3(\xi_1,\xi_3,\xi_2+\xi_4)-\sigma_3(-\xi_2,-\xi_4,\xi_2+\xi_4)]\xi_4|}{|h_4|}\nonumber\\
&&+\frac{|[\sigma_3(-\xi_2,-\xi_4,\xi_2+\xi_4)-\sigma_3^{-}(-\xi_2,-\xi_4,\xi_2+\xi_4)]\xi_4|}{|\epsilon\beta_{\alpha,4}|},
\end{eqnarray}
and \eqref{eq:diffm3} for the second term,  \eqref{eq:mvt} if
$N_{12}\ll N_3$ (in this case, $N_3\sim N_4$), and \eqref{eq:m3} if
$N_{12}\ges N_3$ for the first term, then we get
\begin{equation}
\frac{|J_1|}{|h_4-\epsilon\beta_{\alpha,4}|}\les \frac{1}{N_1^4}.
\end{equation}
The term $J_2$ is identical to the term $J_1$. Now we consider
$J_3$. We first assume that $N_{12}\ges N_3$. Then by the symmetry
of $\sigma_3$, we get
\begin{eqnarray}
J_3&=&[\sigma_3(\xi_1,\xi_3,\xi_2+\xi_4)-\sigma_3^{-}(-\xi_2,-\xi_4,\xi_2+\xi_4)\nonumber\\
&&\quad+\sigma_3(\xi_1,\xi_4,\xi_2+\xi_3)-\sigma_3^{-}(-\xi_2,-\xi_3,\xi_2+\xi_3)]\xi_2\nonumber\\
&=&[\sigma_3(\xi_1,\xi_3,\xi_2+\xi_4)-\sigma_3(-\xi_2-\xi_3,\xi_3,\xi_2)\nonumber\\
&&\quad+\sigma_3(\xi_1,\xi_4,\xi_2+\xi_3)-\sigma_3(-\xi_2-\xi_4,\xi_4,\xi_2)]\xi_2.
\end{eqnarray}
From \eqref{eq:mvt} and $N_{12}\ges N_3$, we get
\begin{equation}
\frac{|J_3|}{|h_4-\epsilon\beta_{\alpha,4}|}\les
\frac{|J_3|}{|h_4|}\les \frac{1}{N_1^4}.
\end{equation}
If $N_{12}\ll N_3$, then $N_3\sim N_4$. We first write
\begin{eqnarray}
J_3&=&[\sigma_3(\xi_1,\xi_3,\xi_2+\xi_4)-\sigma_3^{-}(\xi_1,\xi_3,\xi_2+\xi_4)\nonumber\\
&&\quad+\sigma_3(-\xi_2,-\xi_3,\xi_2+\xi_3)-\sigma_3^{-}(-\xi_2,-\xi_3,\xi_2+\xi_3)]\xi_2\nonumber\\
&&+[\sigma_3^{-}(-\xi_2,\xi_3,\xi_2+\xi_4)-\sigma_3^{-}(-\xi_2,-\xi_4,\xi_2+\xi_4)\nonumber\\
&&\quad+\sigma_3(\xi_1,\xi_4,\xi_2+\xi_3)-\sigma_3(\xi_1,-\xi_3,\xi_2+\xi_3)]\xi_2\nonumber\\
&&+[\sigma_3^{-}(\xi_1,\xi_3,\xi_2+\xi_4)-\sigma_3^{-}(-\xi_2,\xi_3,\xi_2+\xi_4)\nonumber\\
&&\quad+\sigma_3(\xi_1,-\xi_3,\xi_2+\xi_3)-\sigma_3(-\xi_2,-\xi_3,\xi_2+\xi_3)]\xi_2\nonumber\\
&=&J_{31}+J_{32}+J_{33}.
\end{eqnarray}
It follows from \eqref{eq:mvt} that
\begin{equation}
\frac{|J_{33}|}{|h_4-\epsilon\beta_{\alpha,4}|}\les
\frac{|J_{33}|}{|h_4|}\les \frac{1}{N_1^4}.
\end{equation}
It remains to bound $J_{31}$ and $J_{32}$. First we consider
$J_{31}$. Since $m^2(\xi_3)=1$, we rewrite $J_{31}$ by
\begin{eqnarray}\label{eq:J31}
J_{31}&=&[\sigma_3(\xi_1,\xi_3,\xi_2+\xi_4)-\sigma_3^{-}(\xi_1,\xi_3,\xi_2+\xi_4)\nonumber\\
&&\quad+\sigma_3(-\xi_2,-\xi_3,\xi_2+\xi_3)-\sigma_3^{-}(-\xi_2,-\xi_3,\xi_2+\xi_3)]\xi_2\nonumber\\
&=&A(\xi_1,\xi_3,\xi_2+\xi_4)(m^2(\xi_1)\xi_1+\xi_3+m^2(\xi_2+\xi_4)(\xi_2+\xi_4))\xi_2\nonumber\\
&&+A(-\xi_2,-\xi_3,\xi_2+\xi_3)(-m^2(\xi_2)\xi_2-\xi_3+m^2(\xi_2+\xi_3)(\xi_2+\xi_3))\xi_2\nonumber\\
&=&[A(\xi_1,\xi_3,\xi_2+\xi_4)-A(-\xi_2,-\xi_3,\xi_2+\xi_3)]\xi_3\xi_2\nonumber\\
&&-[A(\xi_1,\xi_3,\xi_2+\xi_4)-A(-\xi_2,-\xi_3,\xi_2+\xi_3)]\xi_2\nonumber\\
&&\times[-m^2(\xi_2)\xi_2+m^2(\xi_2+\xi_3)(\xi_2+\xi_3)]\nonumber\\
&&+A(\xi_1,\xi_3,\xi_2+\xi_4)\xi_2\nonumber\\
&&\times[m^2(\xi_1)\xi_1+m^2(\xi_2+\xi_4)(\xi_2+\xi_4)-m^2(\xi_2)\xi_2+m^2(\xi_2+\xi_3)(\xi_2+\xi_3)]
\end{eqnarray}
where
\[A(\xi_1,\xi_2,\xi_3)=\frac{2\epsilon (|\xi_1|^{2\alpha}+|\xi_2|^{2\alpha}+|\xi_3|^{2\alpha})}{|\xi_1\xi_2\xi_3|^2+\epsilon^2(|\xi_1|^{2\alpha}+|\xi_2|^{2\alpha}+|\xi_3|^{2\alpha})^2}.\]
It's easy to see that $A(\xi_1,\xi_2,\xi_3)$ satisfies
\begin{equation}
|\partial_{\xi_i}A(\xi_1,\xi_2,\xi_3)|\les
\frac{|A(\xi_1,\xi_2,\xi_3)|}{|\xi_i|},\quad i=1,2,3.
\end{equation}
For the first two terms in \eqref{eq:J31} we use \eqref{eq:mvt} by
writing
\begin{eqnarray*}
&&A(\xi_1,\xi_3,\xi_2+\xi_4)-A(-\xi_2,-\xi_3,\xi_2+\xi_3)\\
&=&A(\xi_1,\xi_3,\xi_2+\xi_4)-A(-\xi_2,\xi_3,\xi_2+\xi_4)\\
&&+A(-\xi_2,\xi_3,\xi_2+\xi_4)-A(-\xi_2,\xi_3,\xi_2+\xi_3).
\end{eqnarray*}
For the third term, we note that
\begin{eqnarray}
&&m^2(\xi_1)\xi_1+m^2(\xi_2+\xi_4)(\xi_2+\xi_4)-m^2(\xi_2)\xi_2+m^2(\xi_2+\xi_3)(\xi_2+\xi_3)\nonumber\\
&=&m^2(\xi_2+\xi_4)(\xi_2+\xi_4)-m^2(\xi_2)\xi_2\nonumber\\
&&-m^2(\xi_2+\xi_3+\xi_4)(\xi_2+\xi_3+\xi_4)+m^2(\xi_2+\xi_3)(\xi_2+\xi_3),
\end{eqnarray}
thus we can apply \eqref{eq:dmvt}. Therefore, we get
\begin{equation}
\frac{|J_{31}|}{|h_4-\epsilon\beta_{\alpha,4}|}\les
\frac{|J_{31}|}{|\epsilon\beta_{\alpha,4}||}\les \frac{1}{N_1^4}.
\end{equation}
Last we consider $J_{32}$. We denote
\begin{eqnarray}
B(\xi_1,\xi_2,\xi_3)&=&\frac{1}{i\xi_1\xi_2\xi_3-\epsilon(|\xi_1|^{2\alpha}+|\xi_2|^{2\alpha}+|\xi_3|^{2\alpha})}-\frac{1}{i\xi_1\xi_2\xi_3}\nonumber\\
&=&\frac{\epsilon(|\xi_1|^{2\alpha}+|\xi_2|^{2\alpha}+|\xi_3|^{2\alpha})}{[i\xi_1\xi_2\xi_3-\epsilon(|\xi_1|^{2\alpha}+|\xi_2|^{2\alpha}+|\xi_3|^{2\alpha})]i\xi_1\xi_2\xi_3}.
\end{eqnarray}
It's easy to see that $B(\xi_1,\xi_2,\xi_3)$ satisfies
\begin{equation}\label{eq:eB}
|\partial_{\xi_i}B(\xi_1,\xi_2,\xi_3)|\les
\frac{|B(\xi_1,\xi_2,\xi_3)|}{|\xi_i|},\quad i=1,2,3.
\end{equation}
Let
\begin{equation}\label{eq:m3kdv}
\tilde{\sigma}_3(\xi_1,\xi_2,\xi_3)=\frac{M(\xi_1,\xi_2,\xi_3)}{i\xi_1\xi_2\xi_3},
\end{equation}
then we can rewrite $J_{32}$ by
\begin{eqnarray}\label{eq:J32}
J_{32}&=&[\sigma_3^{-}(-\xi_2,\xi_3,\xi_2+\xi_4)-\sigma_3^{-}(-\xi_2,-\xi_4,\xi_2+\xi_4)\nonumber\\
&&\quad+\sigma_3(\xi_1,\xi_4,\xi_2+\xi_3)-\sigma_3(\xi_1,-\xi_3,\xi_2+\xi_3)]\xi_2\nonumber\\
&=&B(-\xi_2,\xi_4,\xi_2+\xi_4)[-m^2(-\xi_2)\xi_2-\xi_4+m^2(\xi_2+\xi_4)(\xi_2+\xi_4)]\xi_2\nonumber\\
&&+B(\xi_1,\xi_4,\xi_2+\xi_3)[m^2(\xi_1)\xi_1+\xi_4+m^2(\xi_2+\xi_3)(\xi_2+\xi_3)]\xi_2\nonumber\\
&&-B(\xi_2,\xi_3,\xi_2+\xi_4)[-m^2(-\xi_2)\xi_2+\xi_3+m^2(\xi_2+\xi_4)(\xi_2+\xi_4)]\xi_2\nonumber\\
&&-B(\xi_1,-\xi_3,\xi_2+\xi_3)[m^2(\xi_1)\xi_1-\xi_3+m^2(\xi_2+\xi_3)(\xi_2+\xi_3)]\xi_2\nonumber\\
&&+[\tilde{\sigma}_3(-\xi_2,\xi_3,\xi_2+\xi_4)-\tilde{\sigma}_3(\xi_1,-\xi_3,\xi_2+\xi_3)\nonumber\\
&&\quad-\tilde{\sigma}_3(-\xi_2,-\xi_4,\xi_2+\xi_4)+\tilde{\sigma}_3(\xi_1,\xi_4,\xi_2+\xi_3)]\xi_2.
\end{eqnarray}
For the first four terms in \eqref{eq:J32}, we can bound them by the
same way as for $J_{31}$, using \eqref{eq:eB} and the symmetry of
$B$ that $B(\xi_1,-\xi_2,\xi_3)=B(-\xi_1,\xi_2,\xi_3)$. For the last
term, it follows from \eqref{eq:m3kdv} and $m^2(\xi_3)=m^2(\xi_4)=1$
that
\begin{eqnarray}
J_L&=&[\tilde{\sigma}_3(-\xi_2,\xi_3,\xi_2+\xi_4)-\tilde{\sigma}_3(\xi_1,-\xi_3,\xi_2+\xi_3)\nonumber\\
&&\quad-\tilde{\sigma}_3(-\xi_2,-\xi_4,\xi_2+\xi_4)+\tilde{\sigma}_3(\xi_1,\xi_4,\xi_2+\xi_3)]\xi_2\nonumber\\
&=&\frac{-m^2(\xi_2)\xi_2+\xi_3+m^2(\xi_2+\xi_4)(\xi_2+\xi_4)}{-\xi_2\xi_3(\xi_2+\xi_4)}\xi_2\nonumber\\
&&\quad
-\frac{-m^2(\xi_2)\xi_2-\xi_4+m^2(\xi_2+\xi_4)(\xi_2+\xi_4)}{\xi_2\xi_4(\xi_2+\xi_4)}\xi_2\nonumber\\
&&+\frac{m^2(\xi_1)\xi_1+\xi_4+m^2(\xi_2+\xi_3)(\xi_2+\xi_3)}{\xi_1\xi_4(\xi_2+\xi_3)}\xi_2\nonumber\\
&&\quad
-\frac{m^2(\xi_1)\xi_1-\xi_3+m^2(\xi_2+\xi_3)(\xi_2+\xi_3)}{-\xi_1\xi_3(\xi_2+\xi_3)}\xi_2.
\end{eqnarray}
Note that there is a cancelation. Therefore,
\begin{eqnarray}\label{eq:JL}
J_L&=&-\frac{\xi_3+\xi_4}{\xi_3\xi_4}\frac{-m^2(\xi_2)\xi_2+m^2(\xi_2+\xi_4)(\xi_2+\xi_4)}{\xi_2(\xi_2+\xi_4)}\xi_2\nonumber\\
&&+\frac{\xi_3+\xi_4}{\xi_3\xi_4}\frac{m^2(\xi_1)\xi_1+m^2(\xi_2+\xi_3)(\xi_2+\xi_3)}{\xi_1(\xi_2+\xi_3)}\xi_2.
\end{eqnarray}
We rewrite \eqref{eq:JL} by
\begin{eqnarray*}
&&-\frac{\xi_3+\xi_4}{\xi_3\xi_4}\frac{-m^2(\xi_2)\xi_2+m^2(\xi_2+\xi_4)(\xi_2+\xi_4)+m^2(\xi_1)\xi_1+m^2(\xi_2+\xi_3)(\xi_2+\xi_3)}{\xi_2(\xi_2+\xi_4)}\xi_2\\
&&+\frac{\xi_3+\xi_4}{\xi_3\xi_4}[m^2(\xi_1)\xi_1+m^2(\xi_2+\xi_3)(\xi_2+\xi_3)][\frac{1}{\xi_1(\xi_2+\xi_3)}+\frac{1}{\xi_2(\xi_2+\xi_4)}]\xi_2.
\end{eqnarray*}
Therefore, we use \eqref{eq:dmvt} for the first term, and
\eqref{eq:mvt} for the second term, and finally we conclude that
\begin{equation}
\frac{|J_L|}{|h_4-\epsilon\beta_{\alpha,4}|}\les
\frac{|J_L|}{|h_4|}\les \frac{1}{N_1^4},
\end{equation}
which completes the proof of the proposition.
\end{proof}

With the estimate of $\sigma_4$, we immediately get the estimate of
$M_5$. We have the same bound as in the KdV case.
\begin{proposition}
If $m$ is of the form \eqref{eq:Imul}, then
\begin{equation}
|M_5(\xi_1,\ldots,\xi_5)|\les
\left[\frac{m^2(N_{*45})N_{45}}{(N+N_1)(N+N_2)(N+N_3)(N+N_{45})}\right]_{sym},
\end{equation}
where
\[N_{*45}=\min(N_1,N_2,N_3,N_{45},N_{12},N_{13},N_{23}).\]
\end{proposition}

So far we have showed that the multipliers $M_i$, $i=3,4,5$ have the
same bounds as for the KdV equation. We list now some propositions.
\begin{proposition}\label{pm5}
Let $w_i(x,t)$ be functions of space-time with Fourier support
$|\xi|\sim N_i$, $N_i$ dyadic. Then
\begin{equation}
\left|\int_0^\delta\int \prod_{i=1}^5w_i(x,t)dxdt\right|\les
\prod_{j=1}^3\norm{w_j}_{F^{1/4}(\delta)}\norm{w_4}_{F^{-3/4}(\delta)}\norm{w_5}_{F^{-3/4}(\delta)}.
\end{equation}
\end{proposition}
\begin{proof}
It follows from the same argument as for the proof of Lemma 5.1 in
\cite{Tao2} with the Proposition \ref{p21}.
\end{proof}

\begin{proposition}
If the associated multiplier $m$ is of the form \eqref{eq:Imul} with
$s=-3/4+$, then
\begin{equation}
\left|\int_0^\delta \Lambda_5(M_5;u_1,\ldots,u_5)dt\right|\les
N^{-\beta}\prod_{i=1}^5\norm{Iu_i}_{F^0(\delta)},
\end{equation}
where $\beta=3+\frac{3}{4}-$.
\end{proposition}
\begin{proof}
This proposition can be proved by following the proof of Lemma 5.2
in \cite{Tao2} and using proposition \ref{pm5}. We omit the details.
\end{proof}

\begin{proposition}
Let $I$ be defined with the multiplier $m$ of the form
\eqref{eq:Imul} and $s=-3/4$. Then
\begin{equation}
|E_I^4(t)-E_I^2(t)|\les \norm{Iu(t)}_{L^2}^3+\norm{Iu(t)}_{L^2}^4.
\end{equation}
\begin{proof}
Since $E_I^4(t)=E_I^2(t)+\Lambda_3(\sigma_3)+\Lambda_4(\sigma_4)$
and the bound for $\sigma_3$, $\sigma_4$ are the same as in the KdV
case, this proposition follows immediately from Lemma 6.1 in
\cite{Tao2}.
\end{proof}
\end{proposition}

We state a variant local well-posedness result which follows from
slight argument in the last section. This is used to iterate the
solution in the I-method.
\begin{proposition}
If $s>-3/4$, then \eqref{eq:kdvb} is uniformly locally well-posed
for data $\phi$ satisfying $I\phi\in L^2(\R)$. Moreover, the
solution exists on a time interval $[0,\delta]$ with lifetime
\begin{equation}
\delta\sim \norm{I\phi}_{L^2}^{-\alpha},\ \alpha>0,
\end{equation}
and the solution satisfies the estimate
\begin{equation}
\norm{Iu}_{F^s(\delta)}\les \norm{I\phi}_{L^2}.
\end{equation}
\end{proposition}

With these propositions and the scaling \eqref{eq:scaling}, we can
show Theorem \ref{t12} by using the same argument in \cite{Tao2}. We
omit the details.

\section{Limit Behavior}
In this section we prove our third result. It is well-known that
\eqref{eq:kdv} is completely integrable and has infinite
conservation laws, and as a corollary one obtains that let $v$ be a
smooth solution to \eqref{eq:kdv}, for any $k\in \Z_{+}$,
\begin{equation}
\sup_{t\in\R}\norm{v(t)}_{H^k}\les \norm{v_0}_{H^k}.
\end{equation}
There are less symmetries for \eqref{eq:kdvb}. We can still expect
that the $H^k$ norm of the solution remains bounded for a finite
time $T>0$, since the dissipative term behaves well for $t>0$. We
already see that for $k=0$ from \eqref{eq:L2law}. Now we prove for
$k=1$ which will suffice for our purpose. We do not pursue for
$k\geq 2$.

Assume $u$ is a smooth solution to \eqref{eq:kdvb}. Let
$H[u]=\int_\R (u_x)^2-\frac{2}{3}u^3+u^2 dx$, then by the equation
\eqref{eq:kdvb} and partial integration
\begin{eqnarray*}
\frac{d}{dt}H[u]&=&\int_\R 2u_x \partial_x(u_t)-2u^2u_t+2uu_t dx\\
&=&\int_\R 2u_x(-u_{xxxx}-\epsilon |\partial_x|^{2\alpha}\partial_xu-(u^2)_{xx})dx\\
&&+\int_\R 2u^2(u_{xxx}+\epsilon |\partial_x|^{2\alpha}u+(u^2)_{x})dx+\int_\R-2\epsilon(\Lambda^\alpha u)^2 dx\\
 &=&\int_\R -2\epsilon (\Lambda^{1+\alpha}u)^2+2\epsilon
u^2\Lambda^{2\alpha}u-2\epsilon(\Lambda^\alpha u)^2 dx\\
&\leq& -\epsilon\int_\R
(\Lambda^{2\alpha}u)^2+2u^2\Lambda^{2\alpha}u dx,
\end{eqnarray*}
where we denote $\Lambda=|\partial_x|$. Thus we have
\begin{equation}
\frac{d}{dt}H[u]+\frac{\epsilon}{2} \norm{\Lambda^{2\alpha}u}_2^2
\les \norm{u}_4^4.
\end{equation}
Using Galiardo-Nirenberg inequality
\[
\norm{u}_3^3 \les \norm{u}_2^{5/2}\norm{u_x}_2^{1/2}, \quad
\norm{u}_4^4 \les \norm{u}_2^{3}\norm{u_x}_2
\]
and Cauchy-Schwarz inequality, we get
\begin{equation}\label{eq:H1law}
\sup_{[0,T]}\norm{u(t)}_{H^1}+\epsilon^{1/2}\brk{ \int_{0}^T
\norm{\Lambda^{2\alpha}u(\tau)}_2^2d\tau}^{1/2}\leq
C(T,\norm{\phi}_{H^1}), \quad \forall\ T>0.
\end{equation}

Assume $u_\epsilon$ is a $L^2$-strong solution to \eqref{eq:kdvb}
obtained in the last section and v is a $L^2$-strong solution to
\eqref{eq:kdv} in \cite{Tao2}, with initial data $\phi_1,\phi_2\in
L^2$ respectively. We still denote by $u_\epsilon, v$ the extension
of $u_\epsilon, v$. From the scaling \eqref{eq:scaling}, we may
assume first that $\norm{\phi_1}_{L^2},\norm{\phi_2}_{L^2}\ll 1$.
Let $w=u_\epsilon-v$, $\phi=\phi_1-\phi_2$, then $w$ solves
\begin{eqnarray}\label{eq:diff}
\left \{
\begin{array}{l}
w_t+w_{xxx}+\epsilon |\partial_x|^{2\alpha}u_\epsilon+(w(v+u_\epsilon))_x=0, t\in \R_{+}, x\in \R,\\
v(0)=\phi.
\end{array}
\right.
\end{eqnarray}
We first view $\epsilon |\partial_x|^{2\alpha}u_\epsilon$ as a
perturbation to the difference equation of the KdV equation, and
consider the integral equation of \eqref{eq:diff}
\begin{equation}
w(x,t)=W_0(t)\phi-\int_0^tW_0(t-\tau)[\epsilon
|\partial_x|^{2\alpha}u_\epsilon+(w(v+u_\epsilon))_x]d\tau, \ t\geq
0.
\end{equation}
Then $w$ solves the following integral equation on $t\in [0,1]$,
\begin{eqnarray}
w(x,t)&=&\psi(t)[W_0(t)\phi-\int_0^tW_0(t-\tau)
\chi_{\R_+}(\tau)\psi(\tau)\epsilon
|\partial_x|^{2\alpha}u_\epsilon (\tau)d\tau\nonumber\\
&&\quad
-\int_0^tW_0(t-\tau)\partial_x(\psi^2(\tau)w(v+u_\epsilon))(\tau)d\tau
].
\end{eqnarray}
By Proposition \ref{p42} and Proposition
\ref{p43},\ref{p45},\ref{p412}, we get
\begin{eqnarray}
\norm{w}_{F^0}\les
\norm{\phi}_{L^2}+\epsilon\norm{u_\epsilon}_{L^2_{[0,2]}\dot{H}_x^{2\alpha}}+\norm{w}_{F^0}(\norm{v}_{F^0}+\norm{u_\epsilon}_{F^0}).
\end{eqnarray}
Since from Theorem \ref{t12} we have
\[\norm{v}_{F^0}\les \norm{\phi_2}_{L^2}\ll 1,\quad \norm{u_\epsilon}_{F^0}\les \norm{\phi_1}_{L^2}\ll 1,\]
then we get that
\begin{equation}
\norm{w}_{F^0}\les
\norm{\phi}_{L^2}+\epsilon\norm{u_\epsilon}_{L^2_{[0,2]}\dot{H}_x^{2\alpha}}.
\end{equation}
From Proposition \ref{p41} and \eqref{eq:H1law} we get
\begin{equation}
\norm{u_\epsilon-v}_{C([0,1], L^2)}\les
\norm{\phi_1-\phi_2}_{L^2}+\epsilon^{1/2}C(\norm{\phi_1}_{H^1},\norm{\phi_2}_{L^2}).
\end{equation}

For general $\phi_1,\phi_2 \in L^2$, using the scaling
\eqref{eq:scaling}, then we immediately get that there exists
$T=T(\norm{\phi_1}_{L^2},\norm{\phi_2}_{L^2})>0$ such that
\begin{equation}\label{eq:limitL2}
\norm{u_\epsilon-v}_{C([0,T], L^2)}\les
\norm{\phi_1-\phi_2}_{L^2}+\epsilon^{1/2}C(T,\norm{\phi_1}_{H^1},\norm{\phi_2}_{L^2}).
\end{equation}
Therefore, \eqref{eq:limitL2} automatically holds for any $T>0$, due
to \eqref{eq:L2law} and \eqref{eq:H1law}.

\begin{proof}[Proof of Theorem \ref{t13}]
For fixed $T>0$, we need to prove that $\forall\ \eta>0$, there
exists $\sigma>0$ such that if $0<\epsilon<\sigma$ then
\begin{equation}\label{eq:limitHs}
\norm{S_T^\epsilon(\varphi)-S_T(\varphi)}_{C([0,T];H^s)}<\eta.
\end{equation}
We denote $\varphi_K=P_{\leq K}\varphi$. Then we get
\begin{eqnarray}
&&\norm{S_T^\epsilon(\varphi)-S_T(\varphi)}_{C([0,T];H^s)}\nonumber\\
&\leq&\norm{S_T^\epsilon(\varphi)-S_T^\epsilon(\varphi_K)}_{C([0,T];H^s)}\nonumber\\
&&+\norm{S_T^\epsilon(\varphi_K)-S_T(\varphi_K)}_{C([0,T];H^s)}+\norm{S_T(\varphi_K)-S_T(\varphi)}_{C([0,T];H^s)}.
\end{eqnarray}
From Theorem \ref{t12} and \eqref{eq:limitL2}, we get
\begin{eqnarray}
\norm{S_T^\epsilon(\varphi)-S_T(\varphi)}_{C([0,T];H^s)}\les
\norm{\varphi_K-\varphi}_{H^s}+\epsilon^{1/2}C(T,K,
\norm{\varphi}_{H^s}).
\end{eqnarray}
We first fix $K$ large enough, then let $\epsilon$ go to zero,
therefore \eqref{eq:limitHs} holds.
\end{proof}

\noindent{\bf Acknowledgment.} Part of the work was finished while
the first named  author was visiting the Department of Mathematics
at the University of Chicago under the auspices of China Scholarship
Council. The authors are grateful to Professor Carlos E. Kenig for
his valuable suggestions. This work is supported in part by the
National Science Foundation of China, grant 10571004; and the 973
Project Foundation of China, grant 2006CB805902, and the Innovation
Group Foundation of NSFC, grant 10621061.

\end{document}